\documentclass[preprint]{elsarticle}
\usepackage{amssymb,latexsym,amsthm}
\usepackage{amsmath,amsfonts}
\usepackage{enumerate}
\usepackage{geometry}
\usepackage{url}
\usepackage{tikz}
\usepackage[british]{babel}

\newcommand{\KK}{\mathbb{K}}

\newcommand{\cS}{\mathcal{S}}

\newcommand{\cL}{\mathcal{L}}

\newcommand{\cM}{\mathcal{M}}

\newcommand{\sh}{\mathrm{sh}}

\newcommand{\diam}{\mathrm{diam}}
\newcommand{\PgL}{\mathrm{P}\Gamma{\mathrm{L}}}

\newcommand{\cG}{\mathcal{G}}

\newcommand{\spin}{\text{\itshape spin}}

\newcommand{\PG}{\mathrm{PG}}

\newcommand{\GL}{\mathrm{GL}}

\newcommand{\Res}{\mathrm{Res}}
\newcommand{\Aut}{\mathrm{Aut}}
\newcommand{\aut}{\mathrm{Stab}}
\newcommand{\ch}{\mathrm{char}}
\newcommand{\rank}{\mathrm{rank}\,}

\newcommand{\cP}{\mathcal{P}}

\newcommand{\GF}{\mathrm{GF}}

\theoremstyle{plain}
\newtheorem{lemma}{Lemma}[section]
\newtheorem{theorem}[lemma]{Theorem}
\newtheorem{thm}{Theorem}
\newtheorem{co}[thm]{Corollary}
\newtheorem{corollary}[lemma]{Corollary}

\newtheorem{prop}[lemma]{Proposition}

\theoremstyle{definition}

\newtheorem{remark}[lemma]{Remark}

\begin{document}
\begin{frontmatter}
\title{On transparent embeddings of point-line geometries}
\author[SI]{Ilaria Cardinali}
\ead{ilaria.cardinali@unisi.it}
\address[SI]{Department of Information Engineering and Mathematics, University of Siena,
Via Roma 56, I-53100, Siena, Italy}
\author[LG]{Luca Giuzzi\corref{cor1}}
\ead{luca.giuzzi@unibs.it}
\address[LG]{DICATAM - Section of Mathematics,
University of Brescia,
Via Branze 53, I-25123, Brescia, Italy}
\cortext[cor1]{Corresponding author. Tel. +39 030 3715739; Fax. +39 030 3615745}
\author[SI]{Antonio Pasini}
\ead{antonio.pasini@unisi.it}
\begin{abstract}
  We introduce the class of transparent embeddings for a point-line geometry $\Gamma = (\cP,\cL)$ as the class of full projective embeddings $\varepsilon$ of $\Gamma$ such that the preimage of any projective line fully contained in $\varepsilon(\cP)$ is a line of $\Gamma$. We will then investigate  the transparency of Pl\"ucker embeddings of projective and polar grassmannians and spin embeddings of half-spin geometries and dual polar spaces of orthogonal type. As an application of our results on transparency, we will derive several Chow-like theorems for polar grassmannians and half-spin geometries.
\end{abstract}
\end{frontmatter}

\section{Introduction}\label{Introduction}
\label{Sec1}
Let $\Gamma =(\cP,\cL)$ be a point-line geometry. As usual, we assume that $\Gamma$ is connected and that
no two distinct lines of $\Gamma$ have more than one point in common. A (full) projective embedding of $\Gamma$ into the projective space $\Sigma = \mathrm{PG}(V)$ of a vector space $V$ is an injective map $\varepsilon$  from the
point-set $\cP$ of $\Gamma$ to the point-set of $\Sigma$ satisfying the following two properties:
(E1) the image of $\varepsilon$ spans $\Sigma$;
(E2) every line of $\Gamma$ is mapped by $\varepsilon$
onto a line of $\Sigma$.

The $\varepsilon$-\emph{image} of $\Gamma$ is the point-line subgeometry $\Gamma^\varepsilon:= ({\cP}^\varepsilon, {\cL}^\varepsilon)$ of $\Sigma$, with ${\cP}^\varepsilon:=\{ \varepsilon(p) : p\in\cP \}$ as the point-set and ${\cL}^\varepsilon:=\{ \varepsilon(\ell) : \ell\in\cL \}$ as the set of lines. The point-set ${\cP}^\varepsilon$ will be called the \emph{support} of $\varepsilon$ and denoted by the symbol $|\varepsilon|$ instead of ${\cP}^\varepsilon$, for short.

We shall denote by $\aut(|\varepsilon|)$ and $\Aut(\varepsilon)$ the set-wise stabilizer of $|\varepsilon|$ in the collineation group $\Aut(\Sigma) = \PgL(V)$ of $\Sigma$ and the set-wise stabilizer of ${\cL}^\varepsilon$ in $\aut(|\varepsilon|)$, respectively.

In view of (E1), (E2) and the connectedness of $\Gamma$, both groups $\aut(|\varepsilon|)$ and $\Aut(\varepsilon)$ act faithfully on $|\varepsilon|$. Hence $\Aut(\varepsilon)$ is isomorphic to a subgroup of the automorphism group $\Aut(\Gamma)$ of $\Gamma$, while the same is false for $\aut(|\varepsilon|)$ in general. Indeed, in general $\aut(|\varepsilon|)$ does not stabilize ${\cL}^\varepsilon$; hence it properly contains $\Aut(\varepsilon)$.

More explicitly, we say that an automorphism $g\in \Aut(\Gamma)$ of $\Gamma$ \emph{lifts to} $\Sigma$ \emph{through} $\varepsilon$ (also, $g$ \emph{lifts to} $\Gamma^\varepsilon$, for short) if there exists a collineation $g^\varepsilon \in \Aut(\varepsilon)$ such that $g^\varepsilon\varepsilon = \varepsilon g$. As $\Aut(\varepsilon)$ acts faithfully on $|\varepsilon|$, the collineation $g^\varepsilon$, if it exists, is uniquely determined by $g$. We call it the \emph{lifting} of $g$ to $\Gamma^\varepsilon$. The automorphisms of $\Gamma$ that lift to $\Gamma^\varepsilon$ form a subgroup $\Aut(\Gamma)_\varepsilon$ of $\Aut(\Gamma)$. Clearly, $\Aut(\varepsilon) \cong \Aut(\Gamma)_\varepsilon$. We say that $\varepsilon$ is \emph{homogeneous} if $\Aut(\Gamma)_\varepsilon = \Aut(\Gamma)$.

In general, ${\cL}^\varepsilon$ is a proper subset of the set of projective lines of $\Sigma$ contained in $|\varepsilon|$. For instance, consider the natural embedding of the symplectic quadrangle $W(3,\KK)$ in a $3$-dimensional projective space $\PG(3,\KK)$. The support of this embedding is the whole point-set of $\PG(3,\KK)$, whence all lines of $\PG(3,\KK)$ are contained in it, but not all of them are lines of $W(3,\KK)$.

If all lines of $\Sigma$ contained in $|\varepsilon|$ belong to ${\cL}^\varepsilon$, then we say that the embedding $\varepsilon$ is \emph{transparent}. Clearly, if $\varepsilon$ is transparent then $\aut(|\varepsilon|) = \Aut(\varepsilon)$, hence $\aut(|\varepsilon|) \cong \Aut(\Gamma)_\varepsilon \leq \Aut(\Gamma)$.

In this paper, we shall mainly focus on transparency. However, for a better understanding of this property, it is convenient to place it in a wider context.
In view of that, we need to state a few preliminary conventions.

For any $p,q\in\cP$, let $d(p,q)$ be the distance between $p$ and $q$ in the collinearity graph of $\Gamma$. Let $\diam(\Gamma)$ be the
diameter of this graph. Note that $\diam(\Gamma)$ might be infinite. This would not change so much of the substance of what we are going to say, but it could cause some complications in our exposition. Thus, in order to make things easier, from now on we assume that $\diam(\Gamma) < \infty$. Under this assumption, we introduce two non-negative integers $\chi^\uparrow_\varepsilon$ and $\chi^\downarrow_\varepsilon$, defined as follows:
\begin{equation}\label{hk}
\left.\begin{array}{rcl}
        \chi^\uparrow_\varepsilon + 1 & := &  \max\{1\leq h \leq \diam(\Gamma) ~:~  \langle\varepsilon(x),\varepsilon(y)\rangle\subseteq |\varepsilon|, ~ \forall x,y\in\cP \text{ with } d(x,y) \leq h\}, \\
 & & \\
\chi^\downarrow_\varepsilon + 1 & := & \min\{ 1\leq  h \leq \diam(\Gamma) ~:~ \langle\varepsilon(x),\varepsilon(y)\rangle\not\subseteq
  |\varepsilon|, ~ \forall x,y\in\cP \text{ with }
  d(x,y)>h\}.
\end{array}\right\}
\end{equation}
Needless to say, in the previous definitions the symbol $\langle \varepsilon(x),\varepsilon(y)\rangle$ stands for the line of $\Sigma$ spanned by $\varepsilon(x)$ and $\varepsilon(y)$. Note also that, since the clause $d(x,y) > \diam(\Gamma)$ is empty, the set
\[\{ 1\leq  h \leq \diam(\Gamma) ~:~ \langle\varepsilon(x),\varepsilon(y)\rangle\not\subseteq
  |\varepsilon|, ~ \forall x,y\in\cP \text{ with } d(x,y)>h\}\]
always contains $\diam(\Gamma)$. Hence the number $\chi^\downarrow_\varepsilon$ is well defined even if $\Gamma$ contains pairs of points $x, y$ at maximal distance such that $\langle\varepsilon(x),\varepsilon(y)\rangle\subseteq|\varepsilon|$. If that is the case, then $\chi^\downarrow_\varepsilon = \diam(\Gamma)-1$.

Clearly, $0\leq \chi^\uparrow_\varepsilon \leq \chi^\downarrow_\varepsilon \leq \diam(\Gamma)-1$. We call $\chi^\uparrow_\varepsilon$ and $\chi^\downarrow_\varepsilon$ respectively the \emph{lower} and \emph{upper} \emph{degree of opacity} of $\varepsilon$ and we say
  that the embedding $\varepsilon$ is $(\chi^\uparrow_\varepsilon, \chi^\downarrow_\varepsilon)$-\emph{opaque}. When $\chi^\uparrow_\varepsilon=\chi^\downarrow_\varepsilon$ we put $\chi_\varepsilon := \chi^\uparrow_\varepsilon = \chi^\downarrow_\varepsilon$ and we call $\chi_\varepsilon$ the \emph{tight degree of opacity} of $\varepsilon$, also saying that $\varepsilon$ is \emph{tightly} $\chi_\varepsilon$-\emph{opaque}.

It is clear that a projective embedding is transparent if and only if it is tightly $0$-opaque. On the other hand, it follows from condition (E1) of the definition of projective embedding that an embedding $\varepsilon$ is tightly $(\diam(\Gamma)-1)$-opaque if and only if $|\varepsilon|$ is the full point-set of $\Sigma$. If this is the case and also $\diam(\Gamma) > 1$, then we say that $\varepsilon$ is \emph{completely opaque}. Note that if $\diam(\Gamma) = 1$ then $0 = \chi^\uparrow_\varepsilon = \chi^\downarrow_\varepsilon = \diam(\Gamma) -1$. In this case $\Gamma$ is a projective space and $\varepsilon$ is an isomorphism.

The following two theorems, to be proved in Sections~\ref{Sec4p} and \ref{Sec5p} respectively, are the main results of this paper. They deal with the transparency or opacity of Pl\"ucker embeddings of projective and polar grassmannians and spin embeddings of half-spin geometries and dual polar spaces of orthogonal type. We refer to Sections~\ref{sec3.1} and \ref{sec3.2} for more details on the geometries and the embeddings considered in these two theorems.

\begin{thm}\label{main thm 1}
Let $\cG_k$ be the $k$-grassmannian of a finite dimensional projective geometry $\cG$ defined over a field (namely a commutative division ring) and let $\cS_k$ be the polar $k$-grassmannian of a polar space $\cS$ of rank $n$ associated to an alternating, quadratic or hermitian form of a finite dimensional vector space defined over a field. Then the following hold:
\begin{enumerate}[1.]
\item\label{main pt1} The Pl\"ucker embedding of $\cG_k$ is transparent for any $k$.
\item\label{main pt2} If $k < n$ then the Pl\"ucker embedding of $\cS_k$ is transparent, except when $\cS$ is of symplectic type (namely, it arises from an alternating form).
\item\label{main pt3} Let $\cS$ be of symplectic type. Then the Pl\"ucker embedding of $\cS_1$ is completely opaque. On the other hand, if  $1 < k < n$ then the Pl\"ucker embedding of $\cS_k$ is $(0,1)$-opaque. Finally, the Pl\"ucker embedding of $\cS_n$ is transparent.
\item\label{main2 pt1} If $\cS$ arises from a hermitian form in a vector space of dimension $2n$, then the Pl\"ucker embedding of $\cS_n$ is transparent.
\end{enumerate}
\end{thm}

\begin{thm}\label{main thm 2}
The following hold:
\begin{enumerate}[1.]
\item\label{main2 pt3} The spin embedding of a half-spin geometry is transparent.
\item\label{main2 pt2} The spin embedding of the dual polar space associated to the orthogonal group $\mathrm{O}(2n+1,\KK)$ is tightly $1$-opaque.
\item\label{main2 pt4} The spin embedding of the dual polar space associated to the orthogonal group $\mathrm{O}^-(2n+2,\KK)$ is transparent.
 \end{enumerate}
\end{thm}

Claim~\ref{main pt1} of Theorem~\ref{main thm 1} has been known since long ago. It appears in Chow \cite{C49}, for instance. We will give a new proof of this claim, quite different from that of \cite{C49}, obtaining it as a corollary of a general lemma on the degrees of opacity of projective embeddings of a very large class of point-line geometries (see Section~\ref{Sec2}, Lemma~\ref{l1}). All remaining claims of Theorem~\ref{main thm 1} follow from the first one, with some additional work in the case of hermitian dual polar spaces (Claim~\ref{main2 pt1}). Indeed the embeddings considered in Claims~\ref{main pt2}--\ref{main2 pt1} are induced by the Pl\"ucker embedding of a projective grassmannian.

Similarly, the embeddings considered in Claims~\ref{main2 pt2} and \ref{main2 pt4} of
Theorem~\ref{main thm 2} are induced by the spin embedding of a half-spin geometry.
The third claim of Theorem~\ref{main thm 2} will be obtained a consequence of the first one.
Claims~\ref{main2 pt3} and \ref{main2 pt2} of Theorem~\ref{main thm 2} will be proved with the help of the above mentioned Lemma~\ref{l1} and one more statement in the same vein as that lemma (Section~\ref{Sec2}, Lemma~\ref{cI}).

The reader might wonder why in Theorem~\ref{main thm 1} the dual polar space $\cS_n$ is considered only when $\cS$ is of symplectic type or associated to a hermitian form in a $2n$-dimensional vector space. This is due to the fact that in all remaining cases the Pl\"ucker mapping, which can still be defined for $\cS_n$, is not a projective embedding. Indeed it maps lines of $\cS_n$ onto pairs of points, conics, quadrics, unitals or other kinds of sets of points, according to the type of $\cS$.

The next corollary immediately follows from Theorems~\ref{main thm 1} and \ref{main thm 2} and the fact that the transparency of an embedding $\varepsilon$ implies the equality $\aut(|\varepsilon|) = \Aut(\varepsilon)$. This corollary incorporates and generalizes some results of Chow~\cite{C49}, which correspond to cases (\ref{tt1}) and (\ref{t3}) below (but only dual polar spaces of symplectic type are considered in \cite{C49}).

\begin{co}\label{co}
Let $\Gamma$ be one of the following point-line geometries:
  \begin{enumerate}[(a)]
\item\label{tt1} The $k$-grassmannian of a finite dimensional projective geometry defined over a field.
\item\label{t2}  The $k$-grassmannian of a polar space of rank $n > k$ associated to a quadratic or hermitian form of a finite dimensional vector space over a field.
\item\label{t3} The dual of a polar space of rank $n$ associated to an alternating or hermitian form of a $2n$-dimensional vector space.
\item\label{tt3}  The dual polar space associated to the orthogonal group $\mathrm{O}^-(2n+2, \KK)$.
\item\label{t4}  A half-spin geometry.
\end{enumerate}
\noindent Let $\varepsilon\colon \Gamma\rightarrow \Sigma$ be the Pl\"ucker embedding if $\Gamma$ is as in (\ref{tt1}), (\ref{t2}) or (\ref{t3}) and the spin embedding in cases (\ref{tt3}) and (\ref{t4}). Then
$\aut(|\varepsilon|)  =  \Aut(\varepsilon)  \cong \Aut(\Gamma)_\varepsilon = \Aut(\Gamma)$.
\end{co}
Note that the equality $\Aut(\Gamma)_\varepsilon = \Aut(\Gamma)$, which says that $\varepsilon$ is homogeneous, has nothing to do with the transparency of $\varepsilon$. It holds in each of the cases considered in Theorems~\ref{main thm 1} and \ref{main thm 2}, no matter if $\varepsilon$ is transparent or not (see Section 4).

Case~\ref{main pt3} of Theorem~\ref{main thm 1} with $1 < k < n$ is not mentioned in Corollary~\ref{co}. Indeed, if $\Gamma = \cS_k$ with $1 < k < n$, the polar space $\cS$ is of symplectic type and $\varepsilon$ is its Pl\"ucker embedding, then $\varepsilon$ is not transparent. Nevertheless, as we shall prove in Section~\ref{Sec6}, the equality $\aut(|\varepsilon|) = \Aut(\varepsilon)$ holds also in this case.

On the other hand, with $\cS$ as above but $k = 1$, the Pl\"ucker embedding $\varepsilon:\cS_1\rightarrow \PG(2n-1,\KK)$ is completely opaque; hence $\aut(|\varepsilon|) = \PgL(2n,\KK) > \mathrm{PSp}(2n,\KK)\cdot\Aut(\KK) = \Aut(\cS)$.

Finally, let $\Gamma$ be the dual of the polar space associated to the orthogonal group $\mathrm{O}(2n+1,\KK)$ and let $\varepsilon:\Gamma\rightarrow\PG(2^n,\KK)$ be its spin embedding. Let $\Gamma'$ be the half-spin geometry of rank $n+1$
associated to $O^+(2n+2,\KK)$ defined over $\KK$. Then $\Gamma$ is a subgeometry (but not a subspace) of $\Gamma'$ and $\varepsilon$ is induced by the spin embedding $\varepsilon'$ of $\Gamma'$. Moreover $|\varepsilon| = |\varepsilon'|$ (see Section~\ref{3.2 Spin}). Consequently, $\aut(|\varepsilon|) = \aut(|\varepsilon'|) = \Aut(\varepsilon') > \Aut(\varepsilon)$. So, $\aut(|\varepsilon|) > \Aut(\varepsilon)$ in this case too. \\

\noindent
{\bf Organization of the paper.} In Section~\ref{Sec2} we shall prove a few general results on opacity of embeddings, to be exploited in the proof of Theorems~\ref{main thm 1} and~\ref{main thm 2}. The geometries and the embeddings considered in those two theorems are described in Sections~\ref{Sec3} and \ref{Sec4p}. Section~\ref{Sec5p} contains the proofs of Theorems~\ref{main thm 1} and~\ref{main thm 2}. Section~\ref{Sec6} is devoted to the proof of the equality $\aut(|\varepsilon|) = \Aut(\varepsilon)$ when $\varepsilon$ is the Pl\"{u}cker embedding of a symplectic $k$-grassmannian, $k > 1$.

\section{A few general results on transparency and opacity}
\label{Sec2}

Throughout this section $\varepsilon:\Gamma\rightarrow\Sigma$ is a projective embedding of a point-line geometry $\Gamma$ with finite diameter $\diam(\Gamma) < \infty$. As stated in the Introduction of this paper, $\chi^\uparrow_\varepsilon$ and $\chi^\downarrow_\varepsilon$ are the lower and upper degrees of opacity of $\varepsilon$.

As we did in the definition of $\chi^\uparrow_\varepsilon$ and $\chi^\downarrow_\varepsilon$, given two distinct points $p, q$ of $\Sigma$ we denote by $\langle p,q\rangle$ the line of $\Sigma$ spanned by $p$ and $q$. More generally, given a set $X$ of points of $\Sigma$, we will denote by $\langle X\rangle$ the subspace of $\Sigma$ spanned by $X$.

\subsection{A sufficient condition for the equality $\chi^\uparrow_\varepsilon = \chi^\downarrow_\varepsilon$ to hold} \label{Sec 2.1}

The following lemma provides a sufficient condition for $\varepsilon$ to admit a tight degree of opacity.

Recall that, as stated in the Introduction,  $\Aut(\Gamma)_\varepsilon$ is the subgroup of $\Aut(\Gamma)$ formed by all automorphisms $g\in \Aut(\Gamma)$ that lift to $\Gamma^\varepsilon$. Also recall that, given a graph $G$, a group of automorphisms of $G$ is said to be \emph{distance-transitive} if it acts transitively on the set of ordered pairs of vertices of $G$ at distance $k$, for every $k = 1, 2,\ldots, \diam(G)$.

Here and henceforth by a \emph{punctured plane} we mean a projective plane minus a point.

\begin{lemma}
\label{l1}
Assume the following:
\begin{enumerate}[{\rm({A}1)}]
\item\label{A1} for any positive integer $k$,  every point $p$ and every line $\ell$ of $\Gamma$, if $d(p,x) = d(p,y) = k$ for two distinct points $x, y \in \ell$, then  $d(p,z) = k-1$ for at most one point $z\in \ell$ and all remaining points of $\ell$ have distance $k$ from $p$;
\item\label{A2} The group $\Aut(\Gamma)_\varepsilon$ acts distance-transitively on the collinearity graph of $\Gamma$;
\item\label{A3} The intersection $\alpha\cap |\varepsilon|$ is not a punctured plane, for any plane $\alpha$ of $\Sigma$.
\end{enumerate}
Then $\chi^\uparrow_\varepsilon = \chi^\downarrow_\varepsilon$.
\end{lemma}
\begin{proof}
Let $d := \diam(\Gamma)$. When $\chi^\uparrow_\varepsilon = d-1$ there is nothing to prove. Thus, we assume that $\chi^\uparrow_\varepsilon < d-1$, namely there are points $a,b\in \Gamma$ such that $\langle \varepsilon(a),\varepsilon(b)\rangle$ is not contained in $|\varepsilon|$. In view of (A\ref{A2}), the same holds true for any pair of points $x, y\in \Gamma$ with $d(x,y) = d(a,b)$. We shall prove that the following property holds:
 \begin{itemize}
 \item[$(*)$] For any $k > \chi^\uparrow_\varepsilon+1$, the line $\langle \varepsilon(x),\varepsilon(y)\rangle$ is not fully contained in $|\varepsilon|$, for any choice of $x$ and $y$ at distance $k$.
 \end{itemize}
The equality $\chi^\uparrow_\varepsilon = \chi^\downarrow_\varepsilon$  is a straightforward consequence of $(*)$ and  the definition of $\chi^\downarrow_\varepsilon$.

We shall prove $(*)$ by induction on $k > \chi^\uparrow_\varepsilon +1$. For $k = \chi^\uparrow_\varepsilon+2$ property $(*)$ immediately follows from the definition of $\chi^\uparrow_\varepsilon$ and hypothesis (A\ref{A2}).

Assume $k > \chi^\uparrow_\varepsilon+2$ and that $(*)$ holds true for any value $\chi^\uparrow_\varepsilon+2,\ldots, k-1$. Take $p,q\in\cP$ such that $d(p,q) = k$ and suppose, by way of contradiction, $\langle \varepsilon(p),\varepsilon(q)\rangle \subseteq |\varepsilon|$. Let $p= p_0\sim p_1\sim \ldots\sim p_{k-1}\sim p_k = q$ be a shortest
path from $p$ to $q$ in the collinearity graph of $\Gamma$. Put $r := p_{k-1}$, for short. Also, let $\ell$ be the line of $\Gamma$ through $r$ and $q$ and put $L := \varepsilon(\ell)$. Then, by the inductive hypothesis, the projective line $M := \langle \varepsilon(p),\varepsilon(r)\rangle$ contains at least one point outside $|\varepsilon|$, while the line joining $\varepsilon(p)$ with $\varepsilon(x)$ for $x\in\ell$ and $x\neq r$ is entirely contained in $|\varepsilon|$, by the hypotheses made on $p$ and $q$ and assumptions (A\ref{A1}) and (A\ref{A2}). It follows that the points of the plane $\alpha := \langle \varepsilon(p), \varepsilon(q),\varepsilon(r)\rangle$ that do not belong to $|\varepsilon|$ are on the line $M$. Thus, $\alpha\setminus (\alpha\cap |\varepsilon|) = \{R_1,\ldots, R_t\}$ for suitable points $R_1,\ldots, R_t\in  M=\langle \varepsilon(p),\varepsilon(r)\rangle$.

Note that $t$ might be infinite but, by assumption (A\ref{A3}),
$t > 1$. Given now a point $X = \varepsilon(x) \in \alpha\setminus(L\cup M)$ and $i =1,2,\ldots,t$,
let $x_i$ be the point of $\ell$ such that $\langle\varepsilon(x),\varepsilon(x_i)\rangle$ meets $M$ in $R_i$. Then $\chi^\uparrow_\varepsilon +1 < d(x,x_i) \neq k$ (by the inductive hypothesis and assumption (A\ref{A2})) while $d(x,y) \in \{1, 2,\ldots, \chi^\uparrow_\varepsilon+1\}\cup\{k, k+1, k+2,\ldots\}$ for any point $y\in \ell\setminus \{x_1,\ldots, x_t\}$. The latter set contains at least two points, namely $r$ and at least one more point $x'$ where $\varepsilon(x') \in \langle\varepsilon(x),\varepsilon(p)\rangle$. Moreover, in view of (A\ref{A1}), there is a number $h$ such that either $d(x,z) = h$ for any $z\in\ell$ or $d(x,z_0) = h-1$ for just one point $z_0\in\ell$ and $d(x,z) = h$ for any point $z\in \ell\setminus \{z_0\}$.

In the situation we are considering, we have $d(x,y) \neq d(x,x_i)$ for any $y\in \ell\setminus\{x_1,\ldots, x_t\}$ and $i = 1, 2,\ldots, t$ (by (A\ref{A2})). As both $t$ and $|\ell\setminus\{x_1,\ldots, x_t\}|$ are greater than $1$, this implies that when $z$ ranges in $\ell$
the number $d(x,z)$ assumes at least two distinct values and each of them is taken at least twice. This contradicts hypothesis (A\ref{A1}).
\end{proof}

The next two corollaries immediately follow from Lemma \ref{l1}.

\begin{corollary}
\label{c0}
  If $\diam(\Gamma) = 2$ and $\Gamma$ and $\varepsilon$ satisfy the hypotheses of Lemma~\ref{l1}, then $\varepsilon$ is either transparent or completely opaque.
\end{corollary}

\begin{corollary}
\label{c00}
Let $\Gamma$ and $\varepsilon$ satisfy the hypotheses of Lemma~\ref{l1}. Then $\varepsilon$ is transparent if and only if $\langle\varepsilon(x),\varepsilon(y)\rangle\not\subseteq|\varepsilon|$ for at least one pair of points $x, y$ of $\Gamma$ at distance $2$.
\end{corollary}

As a direct consequence of Corollary~\ref{c0} we get back the following well known fact: if $\Gamma$ is a classical polar space and $\varepsilon$ embeds $\Gamma$ in $\Sigma = \PG(V)$ as the polar space $\Gamma^\varepsilon$ associated to a sesquilinear or pseudoquadratic form of $V$, then either $\varepsilon$ is transparent or $\Gamma^\varepsilon$ is associated to an alternating form.

\subsection{From local to global transparency}
\label{Sec 2.2}

We recall that a \emph{subspace} of $\Gamma$ is a set $S$ of points of $\Gamma$ such that if a line $\ell$ of $\Gamma$ meets $S$ in at least two points then $\ell\subseteq S$. Subspaces are often regarded as subgeometries, by taking as lines of a subspace $S$ of $\Gamma$ the lines of $\Gamma$ contained in $S$. A subspace is said to be \emph{connected} if it is connected as a point-line geometry.

Given a connected nonempty subspace $S$ of $\Gamma$, let $\Sigma_S := \langle \varepsilon(S)\rangle$ be the span of $\varepsilon(S)$ in $\Sigma$ and let $\varepsilon_S$ be the restriction of $\varepsilon$ to $S$, regarded as a mapping from $S$ to $\Sigma_S$. Then $\varepsilon_S:S\rightarrow\Sigma_S$ is an embedding of $S$, the latter being regarded as a subgeometry of $\Gamma$.

Assume that $\Gamma$ admits a nonempty family $\mathfrak{S}$ of connected nonempty subspaces of $\Gamma$, satisfying the following two conditions:
\newcounter{rem}
\begin{enumerate}[(S1)]
\item\label{S1} $\Aut(\Gamma)_\varepsilon$ stabilizes $\mathfrak{S}$ and acts transitively on it.
\item\label{S2} For $S \in \mathfrak{S}$, we have $|\varepsilon_S| = \Sigma_S\cap |\varepsilon|$.
\setcounter{rem}{\value{enumi}}
\end{enumerate}
In view of (S\ref{S1}), all members of $\mathfrak{S}$ have the same diameter, say $d_0$.
Moreover, still by (S\ref{S1}),  for any choice of $S$ and $S'$ in the same family $\mathfrak{S}$ there exists $g\in \Aut(\Gamma)_\varepsilon$ such that $S' = g(S)$ and $\varepsilon_{S'} = g^\varepsilon\varepsilon_S$. It follows that $\varepsilon_S$ and $\varepsilon_{S'}$ have the same opacity degrees, say  $\chi^\uparrow_0 := \chi^\uparrow_{\varepsilon_S} = \chi^\uparrow_{\varepsilon_{S'}}$ and $\chi^\downarrow_0 := \chi^\downarrow_{\varepsilon_S} = \chi^\downarrow_{\varepsilon_{S'}}$.

\begin{lemma}
\label{cI}
Suppose that (S\ref{S1}) and (S\ref{S2}) hold for a nonempty family $\mathfrak{S}$ of connected nonempty subspaces of $\Gamma$. Assume also the following:

\begin{enumerate}[\rm{(B}1\rm{)}]
\item\label{B1} The group $\Aut(\Gamma)_\varepsilon$ acts distance transitively on the collinearity graph of $\Gamma$.
\item\label{B2} The embedding $\varepsilon$ admits a tight degree of opacity $\chi_\varepsilon$.
\item\label{B3}  With $d_0$, $\chi^\uparrow_0$ and $\chi^\downarrow_0$ as above, we have $\chi_0^\uparrow = \chi_0^\downarrow < d_0-1$. In other words, $d_0 > 1$ and the members of $\mathfrak{S}$ admits a tight degree of opacity $\chi_0 < d_0-1$.
\end{enumerate}
Then $\chi_\varepsilon = \chi_0$.
\end{lemma}
\begin{proof}
Put $\chi_0 := \chi^\uparrow_0 = \chi^\downarrow_0$ (see (B\ref{B3})). Let us firstly prove that $\chi_\varepsilon \leq \chi_0$. To this goal, we must check that, for any choice of two points $p, q\in\cP$ at distance $\chi_0+1$, the line $\langle \varepsilon(p),\varepsilon(q)\rangle$
is not fully contained in $|\varepsilon|$. Suppose the contrary: let $d(p,q) = \chi_0+1$ such that $L\: = \langle \varepsilon(p),\varepsilon(q)\rangle\subseteq |\varepsilon|$.
Since $d_0 > \chi_0+1$ by (B\ref{B3}) and, according to (B\ref{B1}), the group $\Aut(\Gamma)_\varepsilon$ acts distance-transitively on $\Gamma$, we can choose $p$ and $q$ in the same subspace $S\in \mathfrak{S}$. However $\chi_0 = \chi^\uparrow_{\varepsilon_S} = \chi^\downarrow_{\varepsilon_S}$. Hence $L \not\subseteq |\varepsilon_S|$. On the other hand, $L\subseteq |\varepsilon|$ by assumption. So, there exists a point $x\not\in S$ such that $\varepsilon(x)\in L$. This contradicts hypothesis (S\ref{S2}). The inequality $\chi_\varepsilon \leq \chi_0$ follows.

In order to prove the converse inequality $\chi_\varepsilon\geq\chi_0$ we must show that for any $d \leq \chi_0+1$ and at least one pair of points $p$ and $q$ at distance $d$, the line $\langle\varepsilon(p),\varepsilon(q)\rangle$ is contained in $|\varepsilon|$.
  By (B\ref{B3}), this holds for $p$ and $q$ in the same member of $\mathfrak{S}$. As $\Aut(\Gamma)_\varepsilon$ acts distance-transitively on $\Gamma$, the same holds for any two points at distance $d$.
\end{proof}

Lemma~\ref{cI} says that, in order to prove that an embedding $\varepsilon$ is transparent, if conditions (S\ref{S1}), (S\ref{S2}), (B\ref{B1}), (B\ref{B2}) and (B\ref{B3}) are satisfied, we may only check if the embedding $\varepsilon_S$ is transparent for $S\in\mathfrak{S}$.

\section{Grassmannians and their automorphism groups}\label{sec3.1}
\label{Sec3}
In this section we recall some information on Grassmann geometries (also called \emph{shadow geometries} in the literature) and state some terminology and notation. We only consider Grassmann geometries obtained from buildings. The reader willing to know more on this topic is referred to Tits \cite[Chapter 11]{T74} (also Scharlau \cite{Schar} and Pasini \cite[Chapter 5]{P94} for generalizations to chamber systems and diagram geometries).

Let $\Delta$ be an irreducible building of rank $n$, with $I:=\{1,2,\ldots,n\}$ as the set of types. We recall that, given a type $k \in I$, the $k$-\emph{shadow} $\sh_k(F)$ of a flag $F$ of $\Delta$ is the set of $k$-elements of $\Delta$ that are incident to $F$. Let $k^\sim$ be the set of types adjacent to $k$ in the Coxeter diagram of $\Delta$. Then the $k$-\emph{grassmannian} $\Delta_k$ of $\Delta$ (also called the $k$-\emph{shadow geometry} of $\Delta$)  is the point-line geometry whose points are the $k$-elements of $\Delta$ and whose lines are the shadows $\sh_k(F)$ for $F$ a flag of type $k^\sim$. Note that different flags of type $k^\sim$ have different $k$-shadows. So, we can freely switch from shadows to flags, whenever this is convenient.

The full automorphism group $\Aut(\Delta)$ of $\Delta$ induces a group of permutations on the type-set $I$. Let $\Aut(\Delta)_k$ be the stabilizer of the type $k$ in $\Aut(\Delta)$. Then $\Aut(\Delta)_k$ acts faithfully on $\Delta_k$ as a subgroup of $\Aut(\Delta_k)$. In all cases we are aware of, actually $\Aut(\Delta_k) = \Aut(\Delta)_k$. This is certainly the case when, as it often happens, there is a way to recover the whole of $\Delta$ from the point-line geometry $\Delta_k$, possibly modulo non-type-preserving automorphisms of $\Delta$.

In the sequel we shall
focus on buildings of Coxeter type $A_n$,  $C_n$ or $D_n$, since these are the only ones involved in our Theorems \ref{main thm 1} and \ref{main thm 2}.
We recall the corresponding diagrams here. The types are the integers that label the nodes of the diagram.

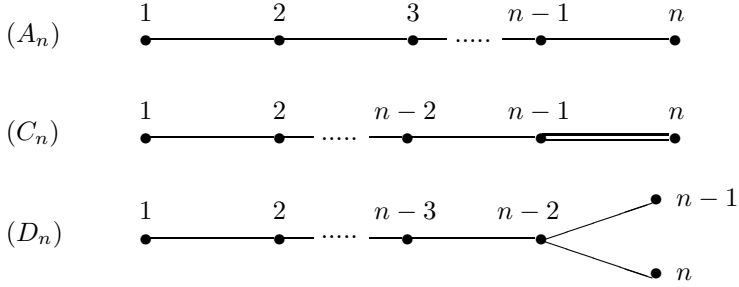
\begin{picture}(310,36)(0,0)
\put(0,10){($A_n$)}
\put(50,8){$\bullet$}
\put(53,11){\line(1,0){47}}
\put(100,8){$\bullet$}
\put(103,11){\line(1,0){47}}
\put(150,8){$\bullet$}
\put(153,11){\line(1,0){12}}
\put(168,10){$.....$}
\put(186,11){\line(1,0){12}}
\put(198,8){$\bullet$}
\put(201,11){\line(1,0){47}}
\put(248,8){$\bullet$}
\put(50,18){$1$}
\put(100,18){$2$}
\put(150,18){$3$}
\put(188,18){$n-1$}
\put(248,18){$n$}
\end{picture}

\begin{picture}(310,36)(0,0)
\put(0,10){($C_n$)}
\put(50,8){$\bullet$}
\put(53,11){\line(1,0){47}}
\put(100,8){$\bullet$}
\put(103,11){\line(1,0){12}}
\put(118,10){$.....$}
\put(136,11){\line(1,0){12}}
\put(148,8){$\bullet$}
\put(151,11){\line(1,0){47}}
\put(198,8){$\bullet$}
\put(201,10){\line(1,0){47}}
\put(201,12){\line(1,0){47}}
\put(248,8){$\bullet$}
\put(50,18){1}
\put(100,18){2}
\put(138,18){$n-2$}
\put(188,18){$n-1$}
\put(248,18){$n$}
\end{picture}

\begin{picture}(310,50)(0,0)
\put(0,23){($D_n$)}
\put(50,21){$\bullet$}
\put(53,24){\line(1,0){47}}
\put(100,21){$\bullet$}
\put(103,24){\line(1,0){12}}
\put(118,24){$.....$}
\put(136,24){\line(1,0){12}}
\put(148,21){$\bullet$}
\put(151,24){\line(1,0){47}}
\put(198,21){$\bullet$}
\put(200,23){\line(3,1){42}}
\put(200,23){\line(3,-1){42}}
\put(241,36){$\bullet$}
\put(241,8){$\bullet$}
\put(50,31){1}
\put(100,31){2}
\put(138,31){$n-3$}
\put(184,31){$n-2$}
\put(251,8){$n$}
\put(251,36){$n-1$}
\end{picture}




\medskip

\noindent
{\bf Further assumptions.} Usually, it is assumed that $n \geq 1$ in $A_n$, $n \geq 2$ in $C_n$ and $n\geq 4$ in $D_n$, but a number of authors also allow $n = 3$ in $D_n$. We will also do so, but we warn that $D_3$ is just the same as $A_3$, except for the way of matching the types with the nodes of the diagram. On the other hand, in order to avoid the rank $1$ case, which is of no interest in the context of this paper, we assume $n \geq 2$ in case $A_n$.

In cases $A_n$ and $D_n$ we also assume that $\Delta$ is \emph{thick}, namely every flag of type $I\setminus\{k\}$ (\emph{cotype} $k$, for short) is incident with at least three $k$-elements, for every type $k \in I$. In case $C_n$ we only assume that $\Delta$ is thick at every type $k < n$, namely every flag of cotype $k$ is incident with at least three $k$-elements and, for type $n$, either $\Delta$ is also thick at $n$
or $\Delta$ is \emph{thin} at $n$, namely every flag of cotype $n$ is incident with just two $n$-elements.

\subsection{Buildings of type $A_n$}\label{An}

It is well known that a thick building $\Delta$ of type $A_n$ is essentially the same object as the system of nonempty proper subspaces of an ordinary (i.e. irreducible) $n$-dimensional projective geometry $\cG$, with symmetrized inclusion as the incidence relation. The $k$-elements of $\Delta$ are the subspaces of $\cG$ of rank $k$ (projective dimension $k-1$). The $1$-grassmannian $\Delta_1$ is just the point-line system of $\cG$  while $\Delta_n$ is the point-line system of the dual $\cG^*$ of $\cG$. For $1 < k<n$ the points of $\Delta_k$ are the subspaces of $\cG$ of rank $k$ and the lines are the $k$-shadows of the flags of type $\{k-1,k+1\}$, namely the sets of the following form:
  \[ \ell_{X,Y}:=\{Z\colon  X \subset Z \subset Y, ~ \rank(Z)=k\}, ~~~  (\rank(X)=k-1, ~ \rank(Y)=k+1).\]
Henceforth, keeping the letter $\cG$ to denote $\Delta$, we shall denote $\Delta_k$ by $\cG_k$. Following a well established custom, we call $\cG_k$ a \emph{projective} $k$-\emph{grassmannian}.

The collineation group $\Aut(\cG)$ of $\cG$ is naturally isomorphic to both $\Aut(\cG_1)$ and $\Aut(\cG_n)$. By a little abuse, $\Aut(\cG) = \Aut(\cG_1) = \Aut(\cG_n)$. Let $1 < k < n$. It has been
known since long ago that we can always recover $\cG$ from $\cG_k$, except, possibly, when $n = 2k+1$ and $\cG$ admits a duality (see Chow \cite{C49}, for instance). However, in the latter case, mistaking $\cG^*$ for $\cG$ is the only error we can make. Consequently $\Aut(\cG) = \Aut(\cG_k)$ for $1 < k < n$, except when $n = 2k-1$ and $\cG$ admits dualities. In the latter case $\Aut(\cG_k)$ also contains the dualities of $\cG$.

Recall that if $n > 2$ then $\cG = \PG(V)$ for an $(n+1)$-dimensional vector space $V$. In this case the subspaces of $\cG$ of rank $k$ are the $k$-dimensional subspaces of $V$. We also recall that $\PG(V)$ admits a duality if and only if the underlying division ring of $V$ is commutative.

\begin{remark}\label{RemAn-1}
  We have distinguished between $\cG$ and $\cG_1$: the latter is a point-line geometry while $\cG$ is the poset of all subspaces of $\cG_1$. However, this distinction is seldom so relevant and many authors neglect it. For the sake of pedantry, a distinction should also be drawn between $\cG$ and $\Delta$. Indeed $\cG$ is an ordered system while $\Delta$ is not. The set of elements of $\Delta$ gets an ordering only via the types. So, while $\cG$ and its dual $\cG^*$ are different objects, no true difference exists between $\Delta$ and its `dual' $\Delta^*$, the latter being the same as $\Delta$ but with types read from right to left. Accordingly, while the automorphisms of $\cG$ are its collineations, both the collineations and the correlations of $\cG$ are taken as automorphisms of $\Delta$.
  In other words, $\Aut(\cG)$ is the group of type-preserving automorphisms of $\Delta$.
\end{remark}

\begin{remark}\label{RemAn-2}
  Only finite dimensional projective geometries have been considered here. Indeed,
  infinite dimensional projective geometries do not fit in the framework we have chosen for this section, since they are not buildings.
  However, the $k$-grassmannian $\cG_k$ of $\cG$ can also be defined when $\cG$ is infinite dimensional, for any positive integer $k$. We will freely refer to this generalization in the next subsection, when considering classical polar spaces naturally embedded in possibly infinite dimensional projective spaces.
\end{remark}

\subsection{Buildings of type $C_n$}\label{Cn}

In this case $\Delta$ is the system of nonempty singular subspaces of a (non-degenerate) polar space $\cS$ of rank $n$, with (symmetrized) inclusion as the incidence relation (See Tits \cite[Chapters 7--9]{T74}, also Buekenhout and Cohen \cite[Chapters 7--10]{BuekCohen}). Recall that a subspace of a point-line geometry
is called \emph{singular} if all of its points are pairwise collinear.

 The $k$-elements of $\Delta$ are the singular subspaces of $\cS$ of rank $k$ (projective dimension $k-1$). The $1$-grassmannian $\Delta_1$ is the point-line system of $\cS$ (which many authors regard as the same object as $\cS$ itself) while $\Delta_n$ is a \emph{dual polar space}, with the maximal and co-maximal singular subspaces of $\cS$ as points and lines respectively, the points of a co-maximal subspace being the maximal subspaces containing it. For $1 < k<n$ the points of $\Delta_k$ are the singular subspaces of $\cS$ of rank $k$ and the lines are the $k$-shadows of the flags of type $\{k-1,k+1\}$. These lines admit just the same description as the lines of a projective grassmannian, except that singular subspaces of $\cS$ are considered instead of subspaces of $\cG$.

Henceforth we write $\cS_k$ instead of $\Delta_k$ and we call $\cS_k$ a \emph{polar $k$-grassmannian}. Characterizations of polar $k$-grassmannians are known for every $k$ (see Cameron \cite{Cam} and Brouwer--Wilbrink \cite{BW} for the case $k = n$ and Hanssens \cite{H}, Hanssens--Thas \cite{HT} for the case $k < n$.). Needless to say, a way to recover $\cS$ from $\cS_k$ is implicit in each of those characterizations.
 A different way, closer to the original
approach by Chow \cite{C49} is offered by Pankov \cite[Chapter 4]{P2010}.


So, $\cS$ can be recovered from $\cS_k$, for any choice of $k$. Consequently, $\Aut(\cS_k) = \Aut(\cS)$ for any $k$. Moreover, $\Aut(\Delta) = \Aut(\cS)$ except when $n = 2$ and $\cS$ is the generalized quadrangle $W(3,\KK)$ associated to an alternating form of $V(4,\KK)$, with $\KK$ a perfect field of characteristic $2$. In that particular case, $\cS$ also admits a duality.

\begin{remark}\label{RemCn-1}
As said above, many authors refuse to distinguish between $\cS$ and $\cS_1$, while others (as Tits \cite{T74}, for instance) prefer to regard $\cS$ as the poset of singular subspaces of $\cS_1$, which in principle is not the same object as $\cS_1$. In the sequel we will be a little sloppy on this point, letting readers free to think of $\cS$ in the way they prefer.
\end{remark}

\subsubsection{Polar spaces thin at top}\label{TopThin}

According to our previous assumptions, $\Delta$ is allowed to be thin at $n$. If this is the case we say that $\cS$ is \emph{thin at top}. As proved in \cite[Chapter 7]{T74}, a polar space $\cS$ is thin at top if and only if $\cS_1$ is either a grid (when $n = 2$) or the $1$-grassmannian of a building of type $D_n$ (when $n > 2$). In particular, when $\cS$ is thin at top each of the lines of the dual polar space $\cS_n$ has exactly two points. So, $\cS_n$ can be regarded as a graph, with its points as the vertices and its lines as the edges. This graph is bipartite. As it is connected, it admits a unique bipartition. It is customary to refer to the two classes of this bipartition as the \emph{two families of maximal singular subspaces} of $\cS$.

\subsubsection{Classical polar spaces}\label{Cn-classical}

A polar space $\cS$ of rank $n$ is said to be \emph{classical} if there exists a vector space $V$ and a non-degenerate sesquilinear form $f$ or a non-singular pseudoquadratic form $q$ such that the singular subspaces of $V$ are precisely the subspaces of $V$ that are totally isotropic for $f$ or totally singular for $q$ respectively, the rank of a singular subspace being its dimension as a subspace of $V$ (See \cite{T74}, also \cite{BuekCohen}.) Assume moreover that $V$ has finite dimension $N$ and let $\cG = \PG(V)$. Recall that $N \geq 2n$ (see \cite{T74} or \cite{BuekCohen}). Accordingly, $\dim(\cG) = N-1 \geq 2n-1$. In this case, for every $k < n$ the lines of $\cS_k$ are lines of $\cG_k$. In short, $\cS_k$ is a \emph{full} subgeometry of $\cG_k$. Turning to $\cS_n$, its lines can be described as follows:
 \[ \ell_X:=\{Y\colon  X \subset Y \subset X^\perp, ~ \rank(Y)=n\}, ~~~ (\rank(X) = n-1)\]
where $X^\perp$ is the set of points of $\cS$ collinear with all points of $X$. In other words, regarded $X$ as a subspace of $V$, $X^\perp$ is the orthogonal of $X$ with respect to the form $f$ (or the sesquilinearization of $q$). In general, the lines of $\cS_n$ are not lines of $\cG_n$. They are contained in lines of $\cG_n$ if and only if $N = 2n$ since in this case $\rank(X^{\perp})=\rank(X)+2$.

In the above setting, the inclusion map of $\cS_1$ in $\cG_1$ is called a \emph{natural embedding} of $\cS_1$. The natural embedding of $\cS_1$ is unique (up to isomorphisms) except when $\cS$ is defined by an alternating form and the underlying field $\KK$ of $V$ has characteristic equal to $2$. In that case $\cS$ can also be regarded as the classical polar space associated to a quadratic form of a vector space $V'$ over the same field $\KK$ as $V$, but with $\dim(V') > \dim(V)$, possibly $\dim(V') = \infty$ (see De Bruyn and Pasini \cite{BP}).  In this case the inclusion map of $\cS_1$ in $\PG(V')$ is another natural embedding of $\cS$. However, this will not cause any ambiguity in this paper. Indeed, all polar spaces to be considered in this paper will be regarded as embedded in a given projective space, via their unique natural embedding (if this is unique) or one of their two natural embeddings, chosen in advance (when two natural embeddings exist).

As already noted at the end of Section \ref{Sec 2.1}, when $\cS$ does not arise from an alternating form of $V$, the lines of $\cG_1$ fully contained in the point-set of $\cS_1$ are just the lines of $\cS_1$, namely the natural embedding of $\cS_1$ is transparent. On  the other hand, when $\cS$ is associated to an alternating form, all points of $\cG_1$ are points of $\cS_1$, whence all lines of $\cG_1$ are fully contained in the point-set of $\cS_1$, but, of course, not all of them are lines of $\cS_1$. In this case the natural embedding of $\cS_1$ is completely opaque.

\begin{remark}\label{RemCn-2}
As proved by Tits \cite[Chapters 7--9]{T74} (see also Buekenhout and Cohen \cite{BuekCohen}), if $n \geq 4$ then $\cS$ is classical. If $n = 3$ then $\cS$ is classical but for two exceptional cases. In one of them $\cS$ is thick, with Moufang non-Desarguesian planes. The interested reader is referred to Tits \cite[Chapter 9]{T74} for a description of this family of polar spaces. In the other case $\cS$ is thin at top and we have $\cS_1 = \cG_2$ where $\cG = \PG(3,\KK)$ for a non-commutative division ring $\KK$.
\end{remark}

\begin{remark}
When $\mathrm{char}(\KK)\neq 2$ the $1$-grassmannian ${\cS}_1$ of a classical polar space $\cS$ admits just one projective embedding, namely the natural one. On the other hand, if $\mathrm{char}(\KK) = 2$ then ${\cS}_1$ can admit many non-natural embeddings, which arise from so-called generalized pseudo-quadratic forms. We refer the interested reader to Pasini \cite{Pas-EPSR} for more on these embeddings.
\end{remark}

\subsubsection{Classical polar spaces defined over a field}\label{Cn-classical-field}

Let $\cS$ be a classical polar space of rank $n$ associated to a sesquilinear form $f$ or a pseudoquadratic form $q$ of a vector space $V$. Assume moreover that the underlying division ring $\KK$ of $V$ is commutative. Then we can always assume that $\cS$ arises from an alternating or hermitian form or a quadratic form (see \cite[Chapter 8]{T74}). Accordingly, we call $\cS$ a \emph{symplectic}, \emph{hermitian} or \emph{orthogonal} polar space, thus referring to its grassmannian as a \emph{symplectic}, respectively \emph{hermitian} or \emph{orthogonal}, grassmannian.

In the rest of this subsection we recall a few well known facts on each of the above cases and fix some notation.
\begin{enumerate}[1.]
\item \emph{Symplectic polar spaces.} Let $\cS$ be symplectic. Then $\dim(V) = 2n$. As remarked in the previous subsection, all points of $\cG = \PG(V)$ are points of $\cS_1$. On the other hand, the symplectic dual polar space $\cS_n$ is a subspace of the $n$-grassmannian $\cG_n$ of $\cG$.

Following a well established custom, we denote $\cS_1$ by the symbol $\mathrm{W}(2n-1,\KK)$ and $\cS_n$ by $\mathrm{DW}(2n-1,\KK)$.

\item \emph{Hermitian polar spaces.}\label{Cn-par2} Let $\cS$ be hermitian. Denoted by $\sigma$ the involutory automorphism of $\KK$ associated to the hermitian form defining $\cS$, the field $\KK$ is a separable quadratic extension of the subfield $\KK_\sigma < \KK$ fixed by $\sigma$. If $L$ is a line of $\cG = \PG(V)$ not contained in the point-set $\cP$ of $\cS_1$, then either $|L\cap\cP| \leq 1$ or $L\cap\cP$ is a Baer subline of $L$, defined over $\KK_\sigma$.
  In general,
  no upper bound can be stated in the Hermitian case
  for $N-2n$, where $N=\dim(V)\geq 2n$.
  However, in this paper we shall mainly focus on the case when $\dim(V)$ is as small as possible, namely $\dim(V) = 2n$. Thus,
  let $\dim(V) = 2n$. Then the lines of $\cS_n$ are Baer sublines of those of $\cG_n$. Actually, $\cS_n$ is a full subgeometry of a Baer subgeometry of $\cG_n$ defined over $\KK_\sigma$. We will turn back to this point in Section~\ref{sec3.2}.

As customary, when $\dim(V) = 2n$ we denote $\cS_1$ and $\cS_n$ by $\mathrm{H}(2n-1,\KK)$ and $\mathrm{DH}(2n-1,\KK)$ respectively.\\

\item\label{Cn-par3}
 \emph{Orthogonal polar spaces.}  When $\cS$ is orthogonal,
 no upper bound can be stated for
 $\dim(V)-2n$ in general,  as in the Hermitian case. We shall only consider the following three cases, where $N = \dim(V)$ takes its three smallest values, namely $2n$, $2n+1$ and $2n+2$. (Note that these are the only possible values for $N$ when $\KK$ is quasi-quadratically closed.)
\begin{enumerate}[a)]
\item $N = 2n$. In this case $\cS$ is thin at top (see Subsection~\ref{TopThin}). The quadratic form $q$ associated to $\cS$ can always be given the following canonical form (modulo a multiplicative factor and the choice of a suitable basis of $V$, of course):
\begin{equation}\label{form O+}
q(x_1,\ldots, x_{2n}) ~ = ~ \sum_{i=1}^nx_{2i-1}x_{2i}.
\end{equation}
If $n = 2$ then $\cS$ is a grid. Leaving this case aside, let $n > 2$.
 If either $n=3$ or $n>4$, then $\Aut(\cS) = \Aut(\mathrm{O}^+(2n,\KK))$.
By $\mathrm{O}^+(2n,\KK)$ we denote the group of all linear
mappings of $V(2n,\KK)$ preserving the quadratic form $q$ as in~\eqref{form O+}.
 If $n = 4$ then $\Aut(\cS)$ has index $3$ in $\Aut(\mathrm{O}^+(2n,\KK))$. In any case, $\mathrm{Inn}(\mathrm{O}^+(2n,\KK)) \leq \Aut(\cS)$. We say that $\cS$ is of type $\mathrm{O}^+(2n,\KK)$ and we denote it by the symbol $\mathrm{Q}^+(2n-1,\KK)$.

\item  $N = 2n+1$. In this case $q$ can be given the following canonical form:
\begin{equation}\label{form O}
q(x_1,\ldots, x_{2n}, x_{2n+1}) ~ = ~ \sum_{i=1}^nx_{2i-1}x_{2i} + x_{2n+1}^2.
\end{equation}
We have $\Aut(\cS) = \Aut(\mathrm{O}(2n+1,\KK))$. We say that $\cS$ is of type $\mathrm{O}(2n+1,\KK)$, also adopting the symbols $\mathrm{Q}(2n,\KK)$ and $\mathrm{DQ}(2n,\KK)$ to denote $\cS_1$ and $\cS_n$ respectively.

Note that the lines of $\cS_n$ appear as conics inside certain projective subspaces of the $n$-grassmannian $\cG_n$ of $\cG = \PG(V)$. As we shall see in Section~\ref{DPP},
the dual polar space $\cS_n$ can also be realized as a subgeometry of the half-spin geometry $\mathrm{HS}(2n+1,\KK)$.

\item\label{Cn-par3c} $N = 2n+2$. Then $q$ can be given the following canonical form, for suitable scalars $\lambda, \mu$, where $\lambda$ and $\mu$ are such that the polynomial $t^2 + \lambda t + \mu$ is irreducible over $\KK$ (hence $\KK$ is not quadratically closed):
\begin{equation}\label{form O-}
q(x_1,\ldots, x_{2n}, x_{2n+1}, x_{2n+2}) ~ = ~ \sum_{i=1}^nx_{2i-1}x_{2i} + x_{2n+1}^2 +\lambda x_{2n+1}x_{2n+2} + \mu x^2_{2n+2}.
\end{equation}
We now have $\Aut(\cS) = \Aut(\mathrm{O}^-(2n+2,\KK))$. We say that $\cS$ is of type $\mathrm{O}^-(2n+2,\KK)$. We use the symbols $\mathrm{Q}^-(2n+1,\KK)$ and $\mathrm{DQ}^-(2n+1,\KK)$ to denote $\cS_1$ and $\cS_n$ respectively.

The lines of $\cS_n$ are elliptic quadrics of projective subspaces of the $n$-grassmannian of $\PG(V)$. As we shall see in Section~\ref{DPP}, the dual polar space $\cS_n$ can also be realized as a subgeometry of the half-spin geometry $\mathrm{HS}(2n+1,\KK(\rho))$ defined over the extension $\KK(\rho)$ of $\KK$, where $\rho$ is any of the roots of the polynomial $t^2 + \lambda t + \mu$.
\end{enumerate}
\end{enumerate}

\subsubsection{Quads of dual polar spaces}\label{quads}

Recall that, given a polar space $\cS$ of rank $n$, the set $\sh_n(X)$ of maximal singular subspaces of $\cS$ containing a given singular subspace $X$ is a subspace of the point-line geometry $\cS_n$. When $\rank(X) = n-2$ the subspace $\sh_n(X)$, regarded as a subgeometry of $\cS_n$, is called a \emph{quad}.

The quads of the dual polar space $\mathrm{DW}(2n-1,\KK)$ are isomorphic to $\mathrm{Q}(4,\KK)$. Conversely, those of $\mathrm{DQ}(2n,\KK)$ are isomorphic to $\mathrm{W}(3,\KK)$. The quads of $\mathrm{DH}(2n-1,\KK)$ are isomorphic to $\mathrm{Q}^-(5,\KK_\sigma)$ while those of $\mathrm{DQ}^-(2n+1,\KK)$ are isomorphic to $\mathrm{H}(3,\KK(\rho))$, with $\KK(\rho)$ as at the end of paragraph~\ref{Cn-par3}.c
of Subsection~\ref{Cn-classical-field}. Finally, the quads of $\mathrm{DQ}^+(2n-1,\KK)$ are dual grids.

\subsection{Buildings of type $D_n$}\label{Dn}

A building $\Delta$ of type $D_n$ is obtained as follows from a polar space $\cS$ of rank $n \geq 3$ thin at top (Tits \cite[Chapter 7]{T74}).
Drop the singular subspaces of $\cS$ of rank $n-1$. The elements of $\Delta$ of type $k < n-1$ are the singular subspaces of $\cS$ of rank $k$. The members of one of the two families of maximal singular subspaces of $\cS$ are taken as elements of $\Delta$ type $n$ and those in the other family as elements of type $n-1$. The incidence relation is the same as in $\cS$, namely inclusion, except when two elements of type $n$ and $n-1$ are involved. In the latter case, two elements $X$ and $Y$ of $\Delta$ of types $n$ and $n-1$ respectively are declared to be incident in $\Delta$ when $X\cap Y$ has rank $n-1$ as a singular subspace of $\cS$.

\subsubsection{Half-spin geometries}\label{Dn n>3}

Let $n > 3$. Then the polar space $\cS$ is of type $\mathrm{O}^+(2n,\KK)$ for some field $\KK$ (Tits \cite[Chapter 7]{T74}). Accordingly, $\Delta$ is said to be \emph{defined over $\KK$} and denoted by the symbol $D_n(\KK)$.

For $k < n-1$, the $k$-grassmannian $\Delta_k$ is just the same as $\cS_k$. In particular, $\cS_1 = \Delta_1$. Since we can recover $\cS$ from $\cS_k$, we can also recover $\Delta$ from $\Delta_k$.

The grassmannians $\Delta_n$ and $\Delta_{n-1}$ have the elements of $\Delta$ of type $n$ or $n-1$ respectively as points and those of type $n-2$ as lines. Characterizations of $\Delta_n$ (equivalently, $\Delta_{n-1}$) are also known (Cooperstein \cite{Coop83}, Shult \cite{S94}). Hence $\Delta$ can also be recovered from either $\Delta_n$ or $\Delta_{n-1}$. Consequently, $\Aut(\Delta)_k = \Aut(\Delta_k)$ for any choice of $k$.

The group $\Aut(\cS) = \Aut(\Delta)_1$ contains elements that swap the two families of maximal singular subspaces of $\cS$, namely permute the types $n$ and $n-1$ of $\Delta$. Hence the grassmannians $\Delta_n$ and $\Delta_{n-1}$ are isomorphic. They are called \emph{half-spin geometries} and denoted by the symbol $\mathrm{HS}(2n-1,\KK)$.

We have $\Aut(\Delta) = \Aut(O^+(2n, \KK))$.
 If $n \neq 4$ then $\Aut(\Delta) = \Aut(\Delta)_k = \Aut(\Delta)_{1,\ldots, n-2}$ for any $k \leq n-2$ while $\Aut(\Delta)_n = \Aut(\Delta)_{n-1}$ is the group of type-preserving automorphisms of $\Delta$. On the other hand, when $n = 4$ the group $\Aut(\Delta)$ also contains elements (called trialities) that cyclically permute the types $1, 3$ and $4$. In this case $\Aut(\Delta)_2 = \Aut(\Delta)$ while each of the groups $\Aut(\Delta)_k$ with $k\in \{1, 3, 4\}$ has index $3$ in $\Aut(\Delta)$.

Turning back to half-spin geometries, suppose firstly that $n > 4$ and consider $\Delta_n$, to fix ideas. The lines of $\Delta_n$ are the $n$-shadows of the elements of $\Delta$ of type $n-2$. The geometry $\Delta_n$ contains two families of maximal projective subspaces, of projective dimension $3$ and $n-1$ respectively. The subspaces of the first family are the $n$-shadows of the elements of $\Delta$ of type $n-3$ while those of the latter are the $n$-shadows of the elements of type $n-1$. Every subspace of the second family contains $3$-subspaces, which correspond to flags of $\Delta$ of type $\{n-4,n-1\}$. We shall call the $3$-subspaces
of the first family \emph{projective $3$-spaces of type $+$} and those of the latter family, corresponding to $\{n-4,n-1\}$-flags, \emph{projective $3$-spaces of type $-$}. Every projective plane of $\Delta_n$ is contained in a number of $3$-subspaces of type $+$ and just one $3$-subspace of type $-$.

The same definitions can be given when $n = 4$, but in this case the projective $3$-spaces of type $-$ are maximal projective subspaces of $\Delta_4$ and the two families of $3$-spaces are permuted by those elements of $\Aut(\Delta_4) = \Aut(\Delta)_4$ which swap the types $1$ and $3$.

\subsubsection{The case $n = 3$}\label{Dn n=3}

Let $n = 3$. A building $\Delta$ of type $D_3$ is the same as a $3$-dimensional projective geometry $\cG$, except that the points, the lines and the planes of $\cG$ are given the types $2, 1$ and $3$ respectively (or $3, 1$ and $2$) instead of $1, 2$ and $3$. Assume that the points of $\cG$ get type $2$ in $\Delta$, to fix ideas. Then the half-spin geometries $\Delta_2$ and $\Delta_3$ are the same as $\cG$ and $\cG^*$ respectively while the polar space $\Delta_1$ is the line-grassmannian $\cG_2$ of $\cG$.

When the underlying division ring $\KK$ of $\cG$ is commutative we denote $\Delta$ by the symbol $D_3(\KK)$.

\subsection{From $D_{n+1}(\KK)$ to $\mathrm{Q}(2n, \KK)$ and $\mathrm{Q}^-(2n+1,\KK_\sigma)$}
 \label{DPP}

It is well known that the polar space $\mathrm{Q}(2n,\KK)$ can be realized as a substructure of the building $D_{n+1}(\KK)$ of type $D_{n+1}$ over $\KK$. When $\KK$ admits an involutory automorphism $\sigma$ we can also obtain $\mathrm{Q}^-(2n+1,\KK_\sigma)$ as a substructure of $D_{n+1}(\KK)$.

\subsubsection{From $D_{n+1}(\KK)$ to $\mathrm{Q}(2n,\KK)$}\label{DPP1}

Let $\Delta = D_{n+1}(\KK)$ be a building of type $D_{n+1}$ over the field $\KK$. Recall that $\Delta_1 = \mathrm{Q}^+(2n+1,\KK)$. Let $V = V(2n+2,\KK)$ and let $H$ be a hyperplane of $\PG(V)$. Assume that $H$ is \emph{non-degenerate}, namely that the form induced by $q$ on $H$  is non-degenerate.
Then $\Delta_1$ induces a polar space $\cS \cong \mathrm{Q}(2n,\KK)$ on $H$.

Explicitly, the singular subspaces of $\cS$ are the singular subspaces of the polar space $\Delta_1$ which are fully contained in $H$. In particular, those of rank $n$ bijectively correspond to the flags $\{X,X'\}$ of $\Delta$ of type $\{n, n+1\}$ such that $X\cap X' \subseteq H$. On the other hand, these flags bijectively correspond to the elements of $\Delta$ of type $n+1$, namely the points of the half-spin geometry $\Delta_{n+1}$. Indeed, $\dim(X\cap H) = n-1$ for every $(n+1)$-element $X$ of $\Delta$ (since $H$ is non-singular by assumption). Hence there exists a unique $n$-element $X'$ of $\Delta$ such that $X\cap X' = X\cap H$.

By the above statements, we can take the elements of $\Delta$ of type $n+1$ as points of the dual polar space $\cS_n$. The lines of the dual polar space $\cS_n$ are singular subspaces of $\cS$ of rank $n-1$, hence they are lines of the spin geometry $\Delta_{n+1}$. Thus, $\cS_n$ is a full subgeometry (but not a subspace) of $\Delta_{n+1}$, with just the same points as $\Delta_{n+1}$, but only a few of the lines of $\Delta_{n+1}$ are lines of $\cS_n$. Indeed, an $(n-1)$-element $X$ of $\Delta$ (namely a line of $\Delta_{n+1}$) is a line of $\cS_n$ if and only if $X\subseteq H$. When $X\not\subseteq H$, the $(n-2)$-element $Y := X\cap H$ of $\Delta$ is a singular subspace of $\cS$ of rank $n-2$. Hence it defines a quad of $\cS_n$, say $Q(Y)$. As noticed in Subsection \ref{quads}, we have $Q(Y) \cong \mathrm{W}(3,\KK)$. On the other hand $Y$ defines a projective $3$-space $S(Y)$ of $\Delta_{n+1}$ of type $+$ (see Subsection~\ref{Dn n>3}). The generalized quadrangle $Q(Y)$ is naturally embedded as $\mathrm{W}(3,\KK)$ in $S(Y)$. The line $X$ of $\Delta_{n+1}$ is a line of $S(Y)$. It appears as a hyperbolic line in $Q(Y)$.

\subsubsection{From $D_{n+1}(\KK)$ to $\mathrm{Q}^-(2n+1,\KK_\sigma)$}\label{DPP2}

Our exposition of this construction is a rephrasing of Carter \cite[Theorem 14.5.2]{Carter}.

Let $\Delta := D_{n+1}(\KK)$ as in the previous subsection, but now we assume that $\KK$ admits a non-trivial involutory automorphism $\sigma$. Let $\KK_0 := \KK_\sigma$ be the subfield of $\KK$ fixed by $\sigma$. Thus $\KK$ is a separable quadratic extension of $\KK_0$, say $\KK = \KK_0(\rho)$ for $\rho\in\KK\setminus\KK_0$. The following is the minimal polynomial of $\rho$:
\[P_\rho(t): ~ = ~ t^2 + \lambda t + \mu, \qquad  (\lambda := -\rho-\rho^\sigma\in \KK_0, ~ \mu := \rho\rho^\sigma \in \KK_0).\]
Let $V = V(2n+2,\KK)$. We can assume that $\Delta_1$ is associated to the quadratic form $q$ of $V$ expressed as follows with respect to a suitable basis  $U = (u_i)_{i=1}^{2n+2}$ of $V$:
\[q(x_1,\ldots, x_{2n+2}) ~ := ~ x_1x_2 + x_3x_4 +\cdots + x_{2n-1}x_{2n} + x_{2n+1}x_{2n+2}.\]
(Compare Subsection \ref{Cn-classical-field}, equation \eqref{form O+}.) Let $\delta$ be the semilinear transformation of $V$ acting as follows:
\[\delta: (x_1, x_2,\ldots, x_{2n}, x_{2n+1}, x_{2n+2}) \mapsto  (x_1^\sigma, x_2^\sigma, \ldots, x_{2n}^\sigma, x_{2n+2}^\sigma, x_{2n+1}^\sigma).\]
Then $\delta$ is an involutory automorphism of $\Delta$ fixing all types $k < n$ and permuting $n$ with $n+1$. In other words, $\delta$ is an automorphism of the polar space $\Delta_1$ swapping the two families of maximal singular subspaces of $\Delta_1$.

Introduce now the following basis of  $V$:
\[U_0 ~:= ~  (u_1, \ldots, u_{2n}, u_{2n+1}+u_{2n+2}, -\rho u_{2n+1} - \rho^\sigma u_{2n+2}).\]
It is easy to see that the points of $\Delta_1$ fixed by $\delta$ are precisely the points of $\Delta_1$ represented by $\KK_0$-linear combinations of vectors of $U_0$. They form a Baer subgeometry $\PG(V_0)$ of $\PG(V)$, where $V_0$ is the $\KK_0$-vector space spanned by $U_0$. Note that the form $q_0$ induced by $q$ on $V_0$ admits the following expression with respect to $U_0$:
\begin{equation}\label{form O- n+1}
q_0(y_1,\ldots, y_{2n+2}) ~ := ~ y_1y_2 + y_3y_4 +\cdots + y_{2n-1}y_{2n} + y_{2n+1}^2 + \lambda y_{2n+1}y_{2n+2} + \mu y_{2n+2}^2
\end{equation}
where $y_1,\ldots, y_{2n+2}$ are coordinates with respect to $U_0$. Comparing (\ref{form O- n+1}) with (\ref{form O-}) of Subsection \ref{Cn-classical-field}, we immediately see that $q_0$ defines in $\PG(V_0)$ an orthogonal polar space $\cS$ of type $\mathrm{O}^-(2n+2,\KK_0)$. A singular subspace $X$ of $\Delta_1$ of rank $k$ is stabilized by $\delta$ if and only if $X\cap\PG(V_0)$ is a singular subspace of $\cS$ of rank $k$. Thus, $\cS$ is canonically isomorphic to the poset $\Delta_{1,\delta}$ of singular subspaces of $\Delta_1$ stabilized by $\delta$. With a harmless abuse, we may assume that $\cS$ and $\Delta_{1,\delta}$ are the same objects.

The dual polar space $\cS_n$ can be embedded in the spin geometry $\Delta_{n+1}$ in a natural way. Indeed the points of $\cS_n$, regarded as maximal elements of $\Delta_{1,\delta}$, bijectively correspond to the flags $\{X,X'\}$ of $\Delta$ of type $\{n,n+1\}$ stabilized by $\delta$, namely such that $X' = X^\delta$. An injective mapping $\eta$ can be defined from the set of $\{n,n+1\}$-flags of $\Delta$ stabilized by $\delta$ to the set of points of $\Delta_{n+1}$ by declaring that, for such a flag $F$, $\eta(F)$ is the $(n+1)$-element of $F$. Thus, we have defined an injective mapping from the set of points of $\cS_n$ to the set of points of $\Delta_{n+1}$. By a little abuse, we also denote the latter mapping by the symbol $\eta$. Denoted by $\cP_n$ the point-set of $\cS_n$, put $\cP_\delta := \eta(\cP_n)$. It is immediate to see that an $(n-1)$-element $X$ of $\Delta$ is stabilized by $\delta$ if and only if, regarded as a line of $\Delta_{n+1}$, it is fully contained in $\cP_\delta$. Thus, $\eta$ is indeed an isomorphism of point-line geometries from $\cS_n$ to $\cP_\delta$, the latter being equipped with the lines of $\Delta_{n+1}$ fully contained in it.

Given an $(n-2)$-element $X$ of $\Delta$ stabilized by $\delta$, namely a singular subspace of $\cS$ of rank $n-2$, let $Q(X)$ be the quad defined by $X$ in $\cS_n$ and $S(X)$ the 3-subspace of type $+$ defined by $X$ in $\Delta_{n+1}$. Then $\eta$ induces a natural embedding of $Q(X) \cong \mathrm{H}(3,\KK)$ in $S(X)\cong \PG(3,\KK)$.

\section{Pl\"{u}cker and spin embeddings}\label{sec3.2}
\label{Sec4p}
\subsection{Pl\"{u}cker embeddings}\label{3.2 Pluecker}

Let $\cG = \PG(V)$ for a vector space $V$ of finite dimension $N$ over a field $\KK$. For $1\leq k < N$, let $\cG_k$ be the $k$-grassmannian of $\cG$.

The Pl\"ucker (or Grassmann) embedding of $\cG_k$ is the map
$\varepsilon^\cG_k:\cG_k\to\PG(\bigwedge^k V)$ sending
every $k$-dimensional subspace $\langle v_1,\ldots,v_k\rangle$ of $V$
to the point $\langle v_1\wedge\cdots\wedge v_k\rangle$ of $\PG(\bigwedge^k V)$.
It is well-known that $\varepsilon^\cG_k$ is a full projective embedding
of $\cG_k$ and its support $|\varepsilon^\cG_k|$ is an algebraic set. In fact $|\varepsilon_k^\cG|$ can be obtained as the intersection of a number of quadrics of $\PG(\bigwedge^kV)$, namely it is described by a set of homogeneous equations of degree 2. In the literature the set $|\varepsilon^\cG_k|$ is called the \emph{$k$-Grassmann variety} of $\cG$.

It is also well known that $\Aut(\cG_k)_{\varepsilon^\cG_k} = \Aut(\cG_k)$, namely $\varepsilon^\cG_k$ is homogeneous. As remarked in Section~\ref{An}, the  group $\Aut(\cG_k)$ either coincides with $\PgL(V)$ or it contains $\PgL(V)$ as a subgroup of index $2$, the latter being the case if and only if $N = 2k$.

Let $\cS$ be a classical polar space of rank $n \leq N/2$, naturally embedded in $\PG(V)$. Let $k\leq n$. As said in Subsection~\ref{Cn-classical}, the point-set $\cP_k$ of the $k$-grassmannian $\cS_k$ of $\cS$ is a subset of the set of points of $\cG_k$. Denoted by $\iota_k^{\cS}$ the inclusion map of $\cP_k$ in $\cG_k$, let $\varepsilon_k^\cS := \varepsilon^\cG_k\cdot\iota_k^\cS$ and put $|\varepsilon^\cS_k| = \varepsilon^\cS_k(\cP_k)$. The map
$\varepsilon^\cS_k$, with the subspace $\langle |\varepsilon^\cS_k|\rangle$ of $\PG(\bigwedge^kV)$ as its codomain, will be called the \emph{Pl\"ucker map} (also \emph{Grassmann map}) of $\cS_k$. The set $|\varepsilon_k^\cS|$ can be described by adding one single tensor equation to the equations that define $|\varepsilon^\cG_k|$ (Pasini \cite{P16}). We call it the $k$-\emph{Grassmann variety} of $\cS$.

As recalled in Subsection \ref{Cn-classical}, if either $k < n$ or $k = n$, $N = 2n$ and $\cS$ is a symplectic polar space, then $\cS_k$ is a full subgeometry of $\cG_k$, namely every line of $\cS_k$ is a line of $\cG_k$. In this case $\varepsilon^{\cS}_k$ is a full projective embedding of $\cS_k$ in $\langle|\varepsilon^{\cS}_k|\rangle$. We call it the  \emph{Pl\"ucker embedding} of $\cS_k$. The group $\Aut(\cS_k) \cong \Aut(\cS)$ is a subgroup of $\Aut(\cG_k)$. Therefore, since $\varepsilon^\cG_k$ is homogeneous, the embedding $\varepsilon_k^\cS$ is also homogeneous.

When $k = n$ and $\cS$ is not of symplectic type, the dual polar space $\cS_n$ is not a full subgeometry of $\cG_n$. In fact, if $N > 2n$ then $\cS_n$ is not even a subgeometry of $\cG_n$. In this case $\varepsilon^\cS_n$ is not a projective embedding.

Finally, let $\cS_n = \mathrm{DH}(2n-1,\KK)$ (Subsection~\ref{Cn-classical-field}, paragraph~\ref{Cn-par2}). Then the lines of $\cS_n$ are proper sublines of lines of $\cG_n$. So, $\cS_n$ is a subgeometry of $\cG_n$, but not a full subgeometry. However, recalling that the quads of $\cS_n$ are isomorphic to $\mathrm{Q}^-(5,\KK_\sigma)$ (where $\KK_\sigma$ is as in Subsection \ref{Cn-classical-field}, paragraph~\ref{Cn-par3}, (b)), it is easy to see that a line of $\cG_n$ contains a line of $\cS_n$ if and only if it contains at least three points of $\cS_n$. It is also known (see e.g. De Bruyn \cite{B1}) that a basis $U$ can be chosen in $\bigwedge^nV$ in such a way that $\langle u\rangle\in |\varepsilon_n^\cS|$ for every $u\in U$ and, if $W$ is the $\KK_\sigma$-span of $U$ in $\bigwedge^nV$, then $|\varepsilon^\cS_n|$ is contained in $\PG(W)$ (hence it spans $\PG(W)$). Moreover, $\varepsilon^\cS_n$ maps every line of $\cS_n$ surjectively onto a line of $\PG(W)$. It follows that $\varepsilon^\cS_n$ is a full projective embedding of $\cS_n$ in $\PG(W)$, henceforth called the \emph{Pl\"{u}cker embedding} of $\cS_n$.

With $W$ as above, let $\cG_n^\circ := (\varepsilon^{\cG}_n)^{-1}(\PG(W))$. Then $\cG^\circ_n$ is the $n$-grassmannian of a Baer subgeometry $\cG^\circ$ of $\cG$ defined over $\KK_\sigma$ and the restriction of $\varepsilon_n^\cG$ to $\cG^\circ_n$, truncated to $\PG(W)$, is the Pl\"{u}cker embedding of $\cG^\circ_n$. Clearly, $\cS_n$ is a full subgeometry of $\cG^\circ_n$ (the lines of $\cS_n$ are lines of $\cG^\circ_n$). Moreover, since a line of $\cG_n$ contains a line of $\cS_n$ if and only if it contains at least three points of $\cS_n$, a line of $\cG^\circ_n$ is a line of $\cS_n$ if and only if it contains at least three points of $\cS_n$, if and only if it is fully contained in $\cS_n$. It is also clear that $\Aut(\cS_n) \leq \Aut(\cG^\circ_n)$. As the Pl\"{u}cker embedding of $\cG^\circ_n$ is homogeneous, the embedding $\varepsilon^\cS_n:\cS_n\rightarrow\PG(W)$ is also homogeneous.

\begin{remark}
Recall that $\dim(\bigwedge^kV) = {N\choose k}$. The codimension $c_k := \mathrm{cod}(\langle|\varepsilon^\cS_k|\rangle)$ of $\langle|\varepsilon^\cS_k|\rangle$ in $\PG(\bigwedge^kV)$ is known in a few cases. For instance, it is well known that if $\cS \cong \mathrm{W}(2n-1,\KK)$ then $c_k = {N\choose{k-2}}$ ($= 0$ when $k = 1$). If $\cS \cong \mathrm{Q}(2n,\KK)$ or $\cS \cong \mathrm{Q}^+(2n-1,\KK)$ and $\mathrm{char}(\KK) \neq 2$ then $c_k = 0$ while if $\mathrm{char}(\KK) = 2$ then $c_k = {N\choose{k-2}}$ (Cardinali and Pasini \cite{CP2}). If $\cS = \mathrm{H}(2n-1,\KK)$ then $c_k = 0$ (Blok and Cooperstein \cite[Theorem 3.1]{BC11}).
\end{remark}

\subsection{Spin embeddings}\label{3.2 Spin}

Let $\Delta = D_n(\KK)$ be the building of type $D_n$ over $\KK$ and $V = V(2^{n-1},\KK)$. The half-spin geometry $\Delta_n = \mathrm{HS}(2n-1,\KK)$ admits a homogeneous full projective embedding $\varepsilon^+_\spin:\Delta_n\rightarrow\PG(V)$ (Chevalley \cite{Chevalley}, also Buekenhout and Cameron \cite{BC}), called the \emph{spin embedding} of $\Delta_n$. The support $|\varepsilon^+_{\spin}|$ of $\varepsilon^+_{\spin}$ is an algebraic set, defined by quadratic equations if $\ch(\KK)\not= 2$ (see~Chevalley \cite[\S III.8.2]{Chevalley}, also Lichtenstein \cite{Licht}, Manivel \cite{Mani}).
We call it the \emph{spinor variety} of $\PG(V)$.

In Subsections~\ref{DPP1} and \ref{DPP2} we have shown how to embed $\mathrm{DQ}(2n,\KK)$ and $\mathrm{DQ}^-(2n+1,\KK_\sigma)$ in $\Delta_{n+1} = \mathrm{HS}(2n+1,\KK)$, where $\Delta$ now stands for $D_{n+1}(\KK)$. Projective embeddings of the dual polar spaces $\mathrm{DQ}(2n+1,\KK)$ and $\mathrm{DQ}^-(2n+1,\KK_\sigma)$ in $\PG(2^n-1,\KK)$ immediately arise from that.

Explicitly, let $\varepsilon^+_\spin$ be the spin embedding of $\Delta_{n+1}$ in $\PG(V)$, $V = V(2^n,\KK)$. Let $\cS = \mathrm{Q}(2n,\KK)$ and let $\iota:\cS_n\rightarrow\Delta_{n+1}$ be the embedding of $\cS_n = \mathrm{DQ}(2n,\KK)$ in $\Delta_{n+1}$ described in Subsection~\ref{DPP1}. Then $\varepsilon_\spin := \varepsilon^+_\spin\cdot \iota$ is a full projective embedding of $\cS_n$ in $\PG(V)$, called the \emph{spin embedding} of $\cS_n$.

As shown in Subsection~\ref{DPP1}, the point-set of $\cS_n$ coincides with the point-set of $\Delta_{n+1}$. Hence the embeddings $\varepsilon^+_\spin$ and $\varepsilon_\spin$ have the same support: $|\varepsilon_\spin| = |\varepsilon_\spin^+|$.

Let $\Aut(\Delta)_{1,\ldots, n+1}$ be the group of type-preserving automorphisms of $\Delta$, regarded as a subgroup of $\Aut(\Delta_1) = \Aut(\Delta)_1$. Then $\Aut(\cS_n) = \Aut(\cS)$ is the stabilizer in $\Aut(\Delta)_{1,\ldots, n+1}$ of the hyperplane $H$ used to construct $\cS$ as a substructure of $\Delta_1$ (see Subsection~\ref{DPP1}). Accordingly, $\Aut(\cS_n)$ can be regarded as a subgroup of $\Aut(\Delta_{n+1})$. Since $\varepsilon^+_\spin$ is homogeneous, the embedding $\varepsilon_\spin$ is homogeneous.

Let now $\cS = \mathrm{Q}^-(2n+1,\KK_\sigma)$, as in Subsection~\ref{DPP2}. Let $\iota^-:\cS_n\rightarrow\Delta_{n+1}$ be the embedding of $\cS_n = \mathrm{DQ}^-(2n+1,\KK_\sigma)$ in $\Delta_{n+1}$
described in Subsection~\ref{DPP2}. Put $\varepsilon^-_\spin := \varepsilon^+_\spin\cdot \iota^-$. The support $|\varepsilon^-_\spin|$ of $\varepsilon^-_\spin$ spans $\PG(V)$ (see e. g. Cooperstein and Shult \cite[2.2]{CS}). Hence $\varepsilon^-_\spin$ is a full projective embedding of $\cS_n$ in $\PG(V)$. It is called the \emph{spin embedding} of $\cS_n$.

The group $\Aut(\cS_n) = \Aut(\cS)$ is the centralizer in $\Aut(\Delta)_{1,n,n+1}$ of the automorphism $\delta$ of $\Delta$ used to construct $\cS$ as a substructure of $\Delta_1$ (see Subsection~\ref{DPP2}). Hence $\Aut(\cS_n)$ is a subgroup of $\Aut(\Delta_{n+1})$. Since $\varepsilon^+_\spin$ is homogeneous, the embedding $\varepsilon^-_\spin$ is also homogeneous.

\begin{remark}
The Pl\"{u}cker embedding $\varepsilon_n^\cS$ of $\cS_n = \mathrm{DQ}(2n,\KK)$ can be also obtained as the composition of $\varepsilon_\spin$ with a rational transformation, constructed by composing the veronesean embedding of $\PG(2^n-1,\KK)$ in the Veronese variety $\cal V$ of $\PG({{2^n+1}\choose 2}-1,\KK)$ with a linear projection of $\cal V$ into the codomain of $\varepsilon_n^\cG$ (Cardinali and Pasini \cite{CP1}). This rational transformation induces a bijection from $|\varepsilon_\spin|$ to $|\varepsilon_n^\cS|$. The set $|\varepsilon_n^\cS|$ is algebraic (Pasini \cite{P16}). A description of $|\varepsilon_\spin| = |\varepsilon_\spin^+|$ as an algebraic set is implicit in this construction too.
\end{remark}

\section{Proof of Theorems \ref{main thm 1} and \ref{main thm 2}}
\label{Sec5p}
\subsection{Preliminaries}\label{Prel}

At the beginning of Section~\ref{Sec3} we recalled a few notions from the theory of buildings. That was enough to introduce grassmannians, but now we need a bit more.

Let $\Delta$ be a thick building of irreducible spherical type over the set of types $I = \{1,2,\ldots, n\}$. Let ${\cal D}(\Delta)$ be the diagram of $\Delta$. Regarded ${\cal D}(\Delta)$ as a graph, a subset of $I$ is said to be \emph{connected} if it is connected as a set of vertices of ${\cal D}(\Delta)$. For $J\subseteq I$ we denote by $J^\sim$ the set of the types $i\in I\setminus J$ which are adjacent in ${\cal D}(\Delta)$ with at least one $j\in J$.

Before continuing, we must better explain what a building is for us. Indeed different equivalent definitions of buildings exist in the literature. According to the earliest one (Tits \cite{T74}), a building is a simplicial complex admitting a family of thin subcomplexes (called apartments) satisfying certain properties (axioms (B2), (B3) and (B4) of \cite[Chapter 3]{T74}). In this perspective, a building can also be regarded as a diagram geometry, with the vertices of the complex as elements and the faces as flags, the maximal simplices (called chambers) being the maximal flags. This diagram-geometric setting is the one we are adopting in this paper.

That being stated, we can turn to residues. Let $\tau$ be the type-function of $\Delta$. Given a non-maximal flag $F$ of $\Delta$, the residue $\Res_\Delta(F)$ of $F$ in $\Delta$ consists of the elements of $\Delta$ of type $j\not\in\tau(F)$ which are incident with $F$, with the incidence relation inherited from $\Delta$. We call $J := I\setminus\tau(F)$ the \emph{type} of $\Res_\Delta(F)$, also saying that $\Res_\Delta(F)$ is a $J$-\emph{residue} of $\Delta$. Note that $F = \emptyset$ is allowed, but $\Res_\Delta(\emptyset) = \Delta$.

Residues are thick buildings of spherical type, the diagram of a $J$-residue $R$ being the diagram induced by ${\cal D}(\Delta)$ on $J$. Clearly, the building $R$ is irreducible if and only if $J$ is connected. In general, given a $J$-residue $R$, many flags $F$ exists such that $\Res_\Delta(F) = R$. However, all of these flags contain a unique flag $F_R$ of type $J^\sim$. In general, $\Res_\Delta(F_R)$ properly contains $R$. If $J$ is connected then $R$ is an irreducible component of $\Res_\Delta(F_R)$.

In Section~\ref{Sec3} we have defined the $k$-grassmannian $\Delta_k$ of $\Delta$. With the terminology stated above, the lines of $\Delta_k$ are the residues of $\Delta$ of type $\{k\}$.

As residues are buildings, grassmannians can be defined for them too. Given a $J$-residue $R$ and a type $k\in J$, we denote by $R_k$ the $k$-grassmannian of $R$. We also denote by $P_k(R)$ the set of points of $R_k$, namely the set of $k$-elements of $R$. In other words, $P_k(R) = \sh_k(F)$ for any flag $F$ such that $\Res_\Delta(F) = R$.

\begin{lemma}\label{Prel0}
Property (A\ref{A1}) of Lemma \ref{l1} holds in the point-line geometry $\Delta_k$, for any $k\in I$. Explicitly, if $p$ and $\ell$ are a point and a line of $\Delta_k$, then either all points of $\ell$ have the same distance from $p$ or $\ell$ contains just one point at minimal distance from $p$, say $d$, all remaining points of $\ell$ being at distance $d+1$ from $p$.
\end{lemma}
\begin{proof} Let $F$ be the flag of type $k^\sim$ such that $\ell = \sh_k(F)$. The projection $H := \mathrm{prj}_{F}(p)$ (see \cite[paragraph 3.19]{T74}) is a flag containing $F$. If $k \in \tau(H)$ the $k$-element of $H$ is the unique point of $\ell$ at minimal distance from $p$. Otherwise, all points of $\ell$ have the same distance from $p$.
\end{proof}
In the sequel we shall sometimes refer to the opposition relation between elements of a building. We refer the reader to \cite[Chapter 3]{T74} for this notion.
We shall also make use of convex hulls of subsets of $\Delta_k$. We recall this notion here. We could state it for an arbitrary point-line geometry, but we prefer to stick to $\Delta_k$.

A subspace $S$ of the point-line geometry $\Delta_k$ is said to be \emph{convex} if, for any two points $p, q\in S$, every
shortest path from $p$ to $q$ in the collinearity graph of $\Delta_k$ is fully contained in $S$. Clearly, the intersection of any family of convex subspaces is a convex subspace.
The \emph{convex hull} $[X]$ of a set $X$ of points of $\Delta_k$ is the smallest convex subspace containing $X$, namely the intersection of all convex subspaces containing $X$.

That being stated, let  $\delta := \diam(\Delta_k)$ be the diameter of the collinearity graph of $\Delta_k$. Consider the following conditions:

\begin{enumerate}[(C1)]
\item\label{C1} There exists a sequence $\{k\} = J_1 \subset J_2 \subset \cdots \subset J_\delta \subseteq I$ of connected subsets of $I$ such that, for every $i = 1, 2,\ldots, \delta$, for any two points $p, q$  of $\Delta_k$ at distance $d(p,q) = i$ in the collinearity graph of $\Delta_k$ and every minimal path $\gamma$ of $\Delta_k$ from $p$ to $q$, there exists a $J_i$-residue $R$ containing all points of $\gamma$.
\item\label{C2} With $p, q$ and $R$ as above, the elements $p$ and $q$ are opposite in the building $R$.
\item\label{C3} With $J_1, J_2,\ldots, J_\delta$ as in (C\ref{C1}) and $1 \leq i \leq \delta$, if $J$ is a connected proper subset of $J_i$ containing $k$ and $R$ is a $J$-residue, then the $k$-grassmannian $R_k$ of $R$ has diameter $\diam(R_k) < i$.
\end{enumerate}
\begin{lemma}\label{Prel1}
Under the hypotheses (C\ref{C1}), (C\ref{C2}) and (C\ref{C3}), given two points $p, q$ of $\Delta_k$ at distance $i$, there exists a unique residue $R_{p,q}$ of type $J_i$ containing both $p$ and $q$. Moreover $P_k(R_{p,q}) = [p,q]$, where $[p,q]$ stands for the convex hull of $\{p,q\}$ in $\Delta_k$.
\end{lemma}
\begin{proof}
  At least one $J_i$-residue $R_{p,q}$ containing both $p$ and $q$ exists by (C\ref{C1}). We shall firstly prove that $R_{p,q}$ is unique. Let $R \neq R_{p,q}$ be another such residue. Buildings satisfy the so-called Intersection Property (Tits \cite[Chapter 12]{T74}, see also Pasini \cite[Chapter 6]{P94} for several equivalent formulations of that property). It follows that $p$ and $q$ are contained in a $J$-residue $R' \subseteq R_{p,q}\cap R$ for a suitable connected proper subset $J$ of $J_i$ containing $k$. Then $d(p,q) < i$ by (C\ref{C3}). This contradicts the hypotheses made on $p$ and $q$. The uniqueness of $R_{p,q}$ is proved
  and (C\ref{C1}) now implies that $[p,q]\subseteq P_k(R_{p,q})$. The converse inclusion follows from (C\ref{C2}). Indeed, as $p$ and $q$ are opposite in $R_{p,q}$, every element of $P_k(R_{p,q})$ is contained in a shortest path of $(R_{p,q})_k$ from $p$ to $q$ (see \cite[Chapter 3]{T74}).
\end{proof}

\begin{corollary}\label{Prel2}
Let (C\ref{C1}), (C\ref{C2}) and (C\ref{C3}) hold. Then $\Aut(\Delta)_k$ acts distance-transitively on the collinearity graph of $\Delta_k$, that is
$\Aut(\Delta)_k$ has Property~(A\ref{A2}) of Lemma~\ref{l1}.
\end{corollary}
\begin{proof} Let $G := \Aut(\Delta)_k$. For $i = 1, 2,\ldots, \delta$, the group $G$ transitively permutes the residues of type $J_i$. Therefore, given two pairs $(p,q)$ and $(p',q')$ of points at distance $i$ in $\Delta_k$, there exists an element $g\in G$ such that $g(p'), g(q') \in P_k(R_{p,q})$. So, we can assume that $p', q' \in P_k(R_{p,q})$. We have $R_{p',q'} = R_{p,q}$ by Lemma \ref{Prel1}. By (C\ref{C2}), the elements $p$ and $q$ are opposite in $R_{p,q}$. The same holds for $p'$ and $q'$. The group of type-preserving automorphism of a thick building of spherical type acts transitively on the set of pairs of opposite elements of the building \cite[Chapter 3]{T74}. Hence the stabilizer of $R_{p,q}$ in $G$ contains elements which map $(p',q')$ onto $(p,q)$.
\end{proof}
It is well known (and straightforward to check) that conditions (C\ref{C1}), (C\ref{C2}) and (C\ref{C3}) hold for any $k$ when $\Delta$ has type $A_n$ (projective grassmannians).
They also hold when $\Delta$ has type $C_n$ and $k = n$ (dual polar spaces) and when $\Delta$ is of type $D_n$ and $k \in \{n-1, n\}$ (half-spin geometries).
Explicitly, the sets $J_1, J_2,\ldots, J_\delta$ can be described as follows:
\begin{enumerate}[1)]
\item $\Delta$ of type $A_n$, $1\leq k \leq n$. Then $\delta = \min\{k, n-k+1\}$. For $i = 1, 2,\ldots, \delta$ we have
\[J_i  ~ =  ~ \{k-i+1, \ldots, k-1,~ k, ~k+1, \ldots, k+i-1\}.\]
\item $\Delta$ of type $C_n$ and $k = n$. In this case $\delta = n$ and
\[J_i ~ = ~ \{n-i+1, \ldots, n-1, n\}, ~~~ \mbox{for}~i = 1, 2,..., n.\]
\item $\Delta$ of type $D_n$ and $k \in \{n-1,n\}$. We now have $\delta = \lfloor n/2\rfloor$ (integral part of $n/2$). Put $k = n$, to fix ideas. Then
$J_1 = \{n\}$, while for $i = 2, 3,\ldots, \delta$ we have
\[J_i ~ = ~ \{n-2i+1, ~n-2i+2, \ldots, n-2, ~n-1, ~n\}.\]
Note that if $n$ is even then $J_\delta = \{1,2,\ldots, n\}$ while if $n$ is odd then $J_\delta = \{2, 3,\ldots, n\}$. This fact is related to the following property of buildings of type $D_n$: two $n$-elements of $\Delta$ at maximal distance in $\Delta_n$ are opposite in $\Delta$ if and only if $n$ is even. When $n$ is odd, the elements of $\Delta$ opposite to a given $n$-element $x$ are the $(n-1)$-elements $y$ characterized by the following property: all $n$-elements of $\sh_n(y)$ have maximal distance from $x$ in $\Delta_n$.
\end{enumerate}
\subsection{Proof of Theorem~\ref{main thm 1}}
\label{Sec4}

Given a vector space $V$ of finite dimension $N$ over a field $\KK$, let $\cG = \PG(V) \cong \PG(N-1,\KK)$ and, for $1\leq k < N$, let $\cG_k$ be the $k$-grassmannian of $\cG$ and let $\varepsilon_k^\cG$ be the Pl\"{u}cker embedding of $\cG_k$ in $\PG(\bigwedge^kV)$. In Chow~\cite[Lemma 5]{C49} it is implicit that $\varepsilon_k^\cG$ is transparent,  but we shall offer a different proof of this fact, exploiting our Lemma~\ref{l1}.

\begin{prop}
  \label{t0}
The embedding $\varepsilon^\cG_k$ is transparent.
\end{prop}
\begin{proof}
Property (A\ref{A1}) of Lemma \ref{l1} holds in $\cG_k$, by Lemma \ref{Prel0}. Put $\varepsilon := \varepsilon_k^\cG$, for short. As remarked in Section~\ref{3.2 Pluecker}, we have $\Aut(\cG)_{\varepsilon} = \Aut(\cG)$. Moreover, properties (C\ref{C1}), (C\ref{C2}) and (C\ref{C3}) of Section~\ref{Prel} hold in $\cG_k$, as remarked at the end of that section. Hence property (A\ref{A2}) of Lemma~\ref{l1} holds for $\varepsilon$, by Corollary \ref{Prel2}. Finally, as remarked in Section~\ref{3.2 Pluecker}, the support $|\varepsilon|$ of $\varepsilon$ is the intersection of a family of quadrics of $\PG(\bigwedge^kV)$. No plane section of such an intersection is a punctured plane. Therefore (A\ref{A3}) of Lemma~\ref{l1} also holds. By that lemma, $\chi^\uparrow_\varepsilon = \chi^\downarrow_\varepsilon = : \chi_\varepsilon$.

Let now $p$ and $q$ be two points of $\cG_k$ at distance $2$. By Lemma~\ref{Prel1}, we have $[p,q] = P_k(R_{p,q})$ for a unique residue $R_{p,q} \cong \PG(3,\KK)$ of type $\{k-1,k,k+1\}$. Clearly, $(R_{p,q})_k \cong \mathrm{Q}^+(5,\KK)$ and $\varepsilon$ induces on $[p,q]$ the natural embedding of $\mathrm{Q}^+(5,\KK)$ in $\langle\varepsilon([p,q])\rangle \cong \PG(5,\KK)$. Moreover, $\varepsilon([p,q]) = \langle \varepsilon([p,q])\rangle\cap |\varepsilon|$. Hence
$\langle \varepsilon(p), \varepsilon(q)\rangle\cap |\varepsilon| = \{\varepsilon(p),\varepsilon(q)\}$. It follows that $\chi^\uparrow_\varepsilon + 1 = 1$. Equivalently $\chi_\varepsilon = 0$, namely $\varepsilon$ is transparent.
\end{proof}
For the rest of this subsection $\cS$ is a polar space of rank $n \geq 2$ associated to a sesquilinear (actually, alternating or hermitian) form $f$ or a quadratic form $q$ of $V$. Recall that $N\geq 2n$.

We shall use the symbol $\perp$ to denote orthogonality with respect to either the form $f$ or the bilinearization of $q$. Given a singular subspace $X$ of $\cS$ of rank $r < n$, the \emph{upper residue} $\Res_\cS^+(X)$ of $X$ is the unique residue of $\cS$ of type $\{r+1,\ldots, n\}$ contained in $\Res_\cS(X)$. It consists of the singular subspaces of $\cS$ which properly contain $X$. Clearly, $\Res_\cS^+(X)$ is essentially the same as the polar space of rank $n-r$ induced by $f$ or $q$ on $X^\perp/X$.

For $1\leq k \leq n$, let $\cS_k$ be the $k$-grassmannian of $\cS$ and $\varepsilon^\cS_k$ the Pl\"{u}cker map of $\cS_k$ into $\PG(\bigwedge^kV)$. As remarked in Section~\ref{3.2 Pluecker}, if $k < n$ then $\varepsilon_k^\cS$ is a projective embedding of $\cS_k$ in $\langle |\varepsilon^\cS_k|\rangle$. The same is true when $k = n$ and $\cS$ is of symplectic type.

The embedding $\varepsilon^\cS_1$ is just the inclusion map of $\cS_1$ in $\PG(V)$. As remarked at the end of Section \ref{Sec 2.1}, the embedding  $\varepsilon_1^\cS$ is transparent except when $\cS$ is of symplectic type. In the latter case $\varepsilon_1^\cS$ is completely opaque.

\begin{prop}
\label{t1}
Let $1 < k < n$. If $\cS$ is not of symplectic type then $\varepsilon_k^\cS$ is transparent. On the other hand, if $\cS$ is a symplectic polar space then $\varepsilon_k^\cS$ is $(0,1)$-opaque.
\end{prop}
\begin{proof}
Put $\varepsilon := \varepsilon_k^\cS$ and $\hat{\varepsilon} := \varepsilon_k^\cG$, for short. Recall that $\hat{\varepsilon}$ induces $\varepsilon$ on $\cS_k$.

Given two distinct points $A, B$ of $\cS_k$, suppose that $\langle \varepsilon(A),\varepsilon(B)\rangle \subseteq |\varepsilon|$. Then $\langle \varepsilon(A),\varepsilon(B)\rangle = \langle\hat{\varepsilon}(A),\hat{\varepsilon}(B)\rangle \subseteq |\hat{\varepsilon}|$. By Proposition~\ref{t0}, the points $A$ and $B$ are collinear in $\cG_k$. Recalling that $A$ and $B$ are subspaces of $\cG$ of rank $k$, put $C := A\cap B$. Then $\rank(C) = k-1$, as $A$ and $B$ are collinear in $\cG_k$. Moreover, $C$ belongs to $\cS$, since both $A$ and $B$ belong to $\cS$. Put $R := \mathrm{Res}_\cS^+(C)$. Then $R$ is a polar space of rank $n-k+1$. The embedding $\varepsilon$ induces on $R$ an embedding isomorphic to the embedding $\varepsilon_R$ of $R$ in $\PG(C^\perp/C)$ as the polar space associated to the form induced on $C^\perp/C$ by the form associated to $\cS$. Indeed the points of $\PG(\bigwedge^kV)$ corresponding to points of the polar space $R$ are precisely those which correspond to subspaces $X \subseteq C^\perp$ containing $C$ and totally isotropic (or totally singular) for the sesquilinear form (respectively, the quadratic form) associated to $\cS$.

It follows that, if $\cS$ is associated with a quadratic or hermitian form, then $\langle \varepsilon(A),\varepsilon(B)\rangle \subseteq |\varepsilon|$ if and only if $A$ and $B$ are collinear as points of $R$. Equivalently, they are collinear in $\cS_k$. In this case $\varepsilon$ is transparent.

Suppose now that $\cS$ is associated to an alternating form $f$. Then $f$ induces an alternating from on $C^\perp/C$. Consequently, $\PG(C^\perp/C) = |\varepsilon_R|$. In this case $\langle\varepsilon(A),\varepsilon(B)\rangle \subseteq |\varepsilon|$, no matter if $A$ and $B$ are collinear or not. However, in any case, $A$ and $B$ have distance $\leq 2$ in $R$. Hence $d(A,B) \leq 2$ in $\cS_k$ too. Therefore $\chi^\downarrow_\varepsilon = 1$. Hence $\chi^\uparrow_\varepsilon \leq 1$. On the other hand, it is easy to see that pairs of points $A$, $B$ of $\cS_k$ also exist which have distance $2$ in $\cS_k$ as well as in $\cG_k$. For such a pair, we have $\langle \varepsilon(A),\varepsilon(B)\rangle = \langle\hat{\varepsilon}(A),\hat{\varepsilon}(B)\rangle\not\subseteq |\hat{\varepsilon}|$ by Proposition~\ref{t0}. It follows that $\chi^\uparrow_\varepsilon = 0$. Hence $\varepsilon$ is $(0,1)$-opaque.
\end{proof}

\begin{remark}
Note that Lemma~\ref{l1} is of no use to prove Proposition~\ref{t1}. Indeed, two distinct families of pairs $(A,B)$ of points at distance $2$ exist in $\cS_k$ when $1 < k < n$, according to whether $A$ and $B$ are collinear in $\cG_k$ or not. The group $\Aut(\cS)$ acts transitively on each of these two families but, clearly, it cannot fuse them in one. So, condition~(A\ref{A2}) of Lemma~\ref{l1} fails to hold.
\end{remark}

\begin{prop}\label{t1-bis}
Assume $N = 2n$ and let $k = n$. Let $\cS$ be associated with an alternating or hermitian form $f$ of $V$. Then the Pl\"{u}cker embedding $\varepsilon_n^\cS$ is transparent.
\end{prop}
\begin{proof}
  We can obtain the conclusion either by Lemma~\ref{l1} with the help of Lemma~\ref{Prel0} and Corollary~\ref{Prel2}, or as a straightforward consequence of Proposition~\ref{t0}, by mimicking the proof
  of Proposition \ref{t1}. We choose the second way, which is easier.

  Suppose firstly that $f$ is alternating. Put $\varepsilon := \varepsilon_n^\cS$. Let $A, B$ be points of $\cS_n$ such that $\langle\varepsilon(A),\varepsilon(B)\rangle\subseteq|\varepsilon|$.
  Proposition~\ref{t0} forces $A$ and $B$ to be collinear in $\cG_n$, namely $\rank(A\cap B) = n-1$. However $A\cap B$ belongs to $\cS$. Hence $A$ and $B$ are collinear in $\cS_n$ too.

The same argument works when $f$ is hermitian, modulo a few minor modifications due to the fact that in this case the codomain of $\varepsilon_n^\cS$ is a Baer subgeometry of $\PG(\bigwedge^nV)$. Still with $\varepsilon = \varepsilon_n^\cS$, put $\hat{\varepsilon} := \varepsilon_n^\cG$. Suppose that $\langle\varepsilon(A),\varepsilon(B)\rangle\subseteq|\varepsilon|$. Then $\langle \varepsilon(A),\varepsilon(B)\rangle$ is a Baer subline of a line $L$ of $\PG(\bigwedge^nV)$. On the other hand, $|\varepsilon| \subseteq |\hat{\varepsilon}|$. Hence $L\cap |\hat{\varepsilon}|$ contains a Baer subline of $L$. However $|\varepsilon_n^\cG|$ is the intersection of a family of quadrics. Therefore, if a line of $\PG(\bigwedge^nV)$ contains at least three points of $|\hat{\varepsilon}|$, that line is fully contained in $|\hat{\varepsilon}|$. It follows that $L \subseteq |\hat{\varepsilon}|$. Hence $A$ and $B$ are collinear in $\cG_n$, by Proposition~\ref{t0}. As above, we obtain that $A$ and $B$ are collinear in $\cS_n$.
\end{proof}
Propositions~\ref{t0}, \ref{t1} and \ref{t1-bis} yield all claims of Theorem~\ref{main thm 1}.

\subsection{Proof of Theorem~\ref{main thm 2}}
\label{Sec5}

\subsubsection{The spin embedding $\varepsilon_\spin^+$ of the half-spin geometry $\mathrm{HS}(2n-1,\KK)$}\label{sec spin HS}

Given a field $\KK$, let $V = V(2^{n-1},\KK)$ and let $\varepsilon := \varepsilon^+_\spin:\Delta_n\rightarrow\PG(V)$ be the spin embedding of the $n$-grassmannian $\Delta_n$ of the building $\Delta = D_n(\KK)$ of type $D_n$ over $\KK$. Recall that $\Aut(\Delta_n)_\varepsilon = \Aut(\Delta_n)$.

\begin{lemma}\label{spin l1}
If either $\mathrm{char}(\KK) \neq 2$ or $\KK$ is infinite, then no plane section of $|\varepsilon|$ is a punctured plane.
\end{lemma}
\begin{proof} As remarked in Section~\ref{3.2 Spin}, the spinor variety $|\varepsilon|$ is an algebraic set. Moreover, if $\mathrm{char}(\KK) \neq 2$ then $|\varepsilon|$ can be obtained as an intersection of quadrics. A plane section of an intersection of quadrics is never a punctured plane.

  Let $\mathrm{char}(\KK) = 2$. Perhaps $|\varepsilon|$ is an intersection of quadrics also in this case, but we could not find this information in the available literature on spinor varieties.
  According to the literature we are aware of, when $\mathrm{char}(\KK) = 2$ we can only claim that $|\varepsilon|$ is an algebraic set. However, when $\KK$ is infinite this is enough to conclude. Indeed, since $|\varepsilon|$ is an algebraic set, given a projective plane $\alpha$ of $\PG(V)$, the intersection $\alpha\cap|\varepsilon|$ is a closed set for the Zariski topology on $\alpha$. On the other hand, a punctured plane is the complement of a point. Singletons are closed sets in the Zariski topology. Theirs complements are open. Thus, if the section $\alpha\cap|\varepsilon|$ is a punctured plane then it is both closed and open in the Zariski topology of $\alpha$. Consequently, the projective plane $\alpha$ would not be connected for the Zariski topology. However, it is well known that every projective space defined over an infinite field is connected for the Zariski topology. It follows that $\KK$ is finite.
\end{proof}

\begin{corollary}\label{spin l1 cor}
Under the hypotheses of Lemma \ref{spin l1}, the embedding $\varepsilon$ admits a tight degree of opacity.
\end{corollary}
\begin{proof} Condition (A\ref{A1}) of Lemma \ref{l1} holds in $\Delta_n$ by Lemma~\ref{Prel0}. Conditions (C\ref{C1}), (C\ref{C2}), (C\ref{C3}) of Section~\ref{Prel} hold for $\Delta_n$. Hence hypothesis (A\ref{A2})
  of Lemma \ref{l1} also holds, by Corollary \ref{Prel2} and since $\Aut(\Delta_n)_\varepsilon = \Aut(\Delta_n) =  \Aut(\Delta)_n$,
see also \cite[Chapter 9]{DRG}. Hypothesis (A\ref{A3}) of Lemma~\ref{l1} holds by Lemma~\ref{spin l1}.
  The conclusion follows from Lemma~\ref{l1}.
\end{proof}
Let $\mathfrak{S}$ be the family of the $n$-grassmannians $R_n := R\cap \Delta_n$ of the residues $R$ of $\Delta$ of type $\{n-3, n-2, n-1, n\}$. Let $S\in\mathfrak{S}$. It is clear from the diagram
 
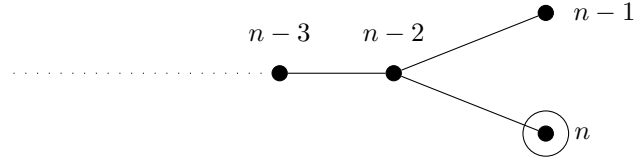
\begin{center}
\begin{tikzpicture}(50,7)
  \tikzstyle{every node}=[draw,circle,fill=black,minimum size=2mm,
label distance=4pt, inner sep=2pt]
 \draw
  (4.5,0) node[label={[yshift=-0.25cm]{$n-3$}}](n-3) {}
  (6,0) node[label={[yshift=-0.25cm]$n-2$}](n-2) {}
  (8,0.8) node[label=right:$ n-1$](n-1) {}
  (8,-0.8) node[label=right:$n$](n) {};

 \draw (n-3)--(n-2)--(n-1);
 \draw (n-2)--(n);
 \draw [loosely dotted] (1.0,0)--(n-3);
 \draw (n) circle [radius=0.3cm];
\end{tikzpicture}
\end{center}
that $S\cong \mathrm{Q}^+(7,\KK)$ and that when $\varepsilon_S$ is the embedding induced by $\varepsilon$ on $S$, $\langle|\varepsilon_S|\rangle\cong \PG(7,\KK)$ and $\varepsilon_S$ is isomorphic to the natural embedding of $\mathrm{Q}^+(7,\KK)$ in $\PG(7,\KK)$. In particular, $\varepsilon_S$ is transparent.

\begin{lemma}\label{spin l2}
Properties (S\ref{S1}) and (S\ref{S2}) of Section~\ref{Sec 2.2} hold for $\mathfrak{S}$.
\end{lemma}
\begin{proof}
Property (S\ref{S1}) follows from the flag-transitivity of the type-preserving automorphism group of $\Delta$. As for (S\ref{S2}), given $S\in \mathfrak{S}$, the equations that define $|\varepsilon_S|$ as an algebraic subset of $\langle|\varepsilon_S|\rangle$ are the restrictions to $\langle |\varepsilon_S|\rangle$ of those that define $|\varepsilon|$ (see \cite{Chevalley}, for instance). The equality $|\varepsilon_S| = \langle|\varepsilon_S|\rangle\cap|\varepsilon|$ follows.
\end{proof}

\begin{prop}\label{spin HS}
Under the hypotheses of Lemma \ref{spin l1}, the embedding $\varepsilon$ is transparent.
\end{prop}
\begin{proof}
This statement immediately follows from Corollary~\ref{spin l1 cor}, Lemma~\ref{spin l2} and Lemma~\ref{cI}, by recalling that $\varepsilon_S$ is transparent for $S\in\mathfrak{S}$.
\end{proof}
The case where $\KK$ is a finite field of characteristic $2$ is not considered in Proposition \ref{spin HS} because, as we have explained in the proof of Lemma \ref{spin l1}, we could not find in the literature anything that clearly implies that hypothesis (A\ref{A3}) of Lemma \ref{l1} holds also in this case. Thus, in order to settle this case, we shall use a different argument.

\begin{prop}\label{spin HS bis}
Let $\KK = \GF(q)$, with $q$ a power of $2$. Then $\varepsilon$ is transparent.
\end{prop}
\begin{proof}
We argue by induction on $n$. If $n = 4$ then $\varepsilon$ is the natural embedding of $\Delta_n = \mathrm{Q}^+(7, q)$ in $\PG(7,q)$. In this case there is nothing to prove.

Let $n > 4$. As in the last paragraph of Section \ref{Prel}, put $\delta = \lfloor n/2\rfloor = \diam(\Delta_n)$ and
\[J_i ~ = ~ \{n-2i+1,~ n-2i+2, \ldots, n-1, n\}\]
for $i = 2, 3,\ldots, \delta$. Put also $r = \delta$ if $n$ is even and $r = \delta+1$ if $n$ is odd, with the additional definition $J_r = \{1, 2,\ldots, n\}$. For $i = 2, 3,\ldots, r$ let $\mathfrak{R}_i$ be the family of $J_i$-residues of $\Delta$ and $\mathfrak{S}_i$ the family of the $n$-grassmannians $R\cap \Delta_n$ for $R\in \mathfrak{R}_i$. Note that $\mathfrak{R}_r = \{\Delta\}$ and $\mathfrak{S}_r = \{\Delta_n\}$.

Recall that, for every $i = 2, 3,\ldots, \delta$ the elements of $\mathfrak{S}_i$ are the convex closures of pairs of points of $\Delta_n$ at distance $i$. If $S\in \mathfrak{S}_i$ then $S \cong \mathrm{HS}(2n_i-1,q)$ where $n_i = n-2i$ and the embedding $\varepsilon_S$ induced by $\varepsilon$ on $S$ is isomorphic to the spin embedding of $\mathrm{HS}(2n_i-1,q)$ in $\PG(2^{n_i-1}-1,q)$. Moreover, $|\varepsilon_S| = \langle|\varepsilon_S|\rangle\cap |\varepsilon|$ (see the proof of Lemma \ref{spin l2}, where the special case $i = 2$ is considered.)

By the inductive hypothesis, if $i < r$ then $\varepsilon_S$ is transparent. By way of contradiction, suppose that $\langle\varepsilon(P),\varepsilon(P')\rangle\subseteq |\varepsilon|$ for two non-collinear points $P$ and $P'$ of $\Delta_n$.  As the convex closure of two points at distance  $i <  r$ is a member of $\mathfrak{S}_i$ and $|\varepsilon_S| = \langle|\varepsilon_S|\rangle\cap |\varepsilon|$ for every $S\in \mathfrak{S}_i$, all points of $\varepsilon^{-1}(\langle\varepsilon(P),\varepsilon(P')\rangle$ have mutually maximal distance. In particular, $d(P,P') = \delta$. Moreover, as the convex closure of two points at distance $\delta$ is a member of $\mathfrak{S}_\delta$ and $\delta < r$ if $n$ is odd, the integer $n$ must be even.
To sum up, the points $X$ of $\Delta_n$ which are not collinear with $P$ and such that $\langle\varepsilon(P),\varepsilon(P')\rangle\subseteq |\varepsilon|$ are contained in the set $\mathrm{Far}(P)$ of the points of $\Delta_n$ at distance $\delta$ from $P$. As $n$ is even, they are the elements of $\Delta$ opposite to $P$.

Let $G$ be the derived group of $\Aut(\Delta)_n = \Aut(\Delta)_\varepsilon$ and let $G_P$ be the stabilizer of $P$ in $G$. It is well known that $G_P$ is transitive on $\mathrm{Far}(P)$. More explicitly, let $U$ be the unipotent radical of $G_P$. Then $U$ acts regularly on $\mathrm{Far}(P)$ (see e.g.~Weiss~\cite[Chapter 11]{Weiss}). Thus, $|\mathrm{Far}(P)| = |U| = q^{n(n-1)/2}$ and the sets $\varepsilon^{-1}(\langle \varepsilon(P),\varepsilon(X)\rangle)\setminus\{P\}$ with $X\in \mathrm{Far}(P)$ form a partition $\pi$ of $\mathrm{Far}(P)$ in $q^{n(n-1)/2-1}$ classes of size $q$. The group $U$ is abelian. In fact,
with a suitable choice of the form
$q$ and $P$,  $U$ can be described as the
group of all matrices of the form
$\begin{pmatrix} I_n & A \\ 0 & I_n\end{pmatrix}$
where $A$ is an arbitrary alternating matrix, as the reader can check. So $U$
it is isomorphic to the additive group $\mathrm{AM}_n(q)$ of $n\times n$ alternating matrices over $\GF(q)$.

Let $u_0\in U$ be such that $\varepsilon(P)\in \langle\varepsilon(u_0(P')),\varepsilon(P')\rangle$ for some $P'\in \mathrm{Far}(P)$. Given any $X\in \mathrm{Far}(P)$, let $u$ be the element of $U$ mapping $P'$ onto $X$. Then $uu_0(P') = u_0u(P') = u_0(X)$, as $U$ is abelian. Therefore
\[\varepsilon(P)  = \varepsilon(u(P)) \in u^\varepsilon(\langle\varepsilon(u_0(P')),\varepsilon(P')\rangle) = \langle\varepsilon(uu_0(P')),\varepsilon(u(P'))\rangle = \langle \varepsilon(u_0(X)),\varepsilon(X)\rangle,\]
where $u^\varepsilon$ is the lifting of $u$ to $\PG(V)$ through $\varepsilon$. In short, if $P'$ and $u_0(P')$ belong to the same class of $\pi$, then the same holds for $X$ and $u_0(X)$, for any $X\in \mathrm{Far}(P)$. By this and the regularity of $U$ on $\mathrm{Far}(P)$ it follows that the stabilizer $U_0$ of all classes of $\pi$ acts regularly on each of those classes. Consequently, $|U_0| = q$.

Let now $C \cong \mathrm{GL}(n,q)$ be a Levi complement of $U$ in $G_P$ (see Carter \cite[Section 8.5]{Carter}). The group $C$ also acts on $\mathrm{Far}(P)$ and stabilizes $\pi$. Since $C$ normalizes $U$, it also normalizes the class-wise stabilizer $U_0$ of $\pi$. However, the action of $C$ on $U$ is isomorphic to the adjoint action of $\mathrm{GL}(n,q)$ on $\mathrm{AM}_n(q)$: for $M\in \mathrm{GL}(n,q)$ and $A\in \mathrm{AM}_n(q)$, the matrix $M$ maps $A$ onto $MAM^T$.

For $0 \leq r \leq \lfloor n/2\rfloor$ let ${\cal A}_r$ be the set of matrices $A\in \mathrm{AM}_n(q)$ such that $\mathrm{rank}(A) = 2r$. It is well known that ${\cal A}_r$ is an orbit of $\mathrm{GL}(n, q)$ in its adjoint action on $\mathrm{AM}_n(q)$. Moreover, ${\cal A}_r$ generates the additive group of $\mathrm{AM}_n(q)$, for every $r = 1, 2,\ldots, \lfloor n/2\rfloor$. It follows that $C$ cannot normalize any non-trivial proper subgroup $U_0$ of $U$.
 \end{proof}
Propositions \ref{spin HS} and \ref{spin HS bis} yield claim 1 of Theorem \ref{main thm 2}.

\begin{remark}
The proof of Proposition \ref{spin HS bis} also works for any field $\KK$, modulo a few obvious minor modifications. Thus, we could fuse Propositions \ref{spin HS} and \ref{spin HS bis} in one statement, to be proved by the same arguments used for Proposition \ref{spin HS bis}. However, while the proof of Proposition \ref{spin HS} can be recycled so as to work for $\mathrm{DQ}(2n,\KK)$ too (see the next subsection, Proposition \ref{spin DQ} and its proof), the proof of Proposition \ref{spin HS bis} can be modified for the same purpose only when $\KK$ is a finite field of even order (see below, Remark \ref{spin DQ rem}). That's why we have preferred to state and prove Propositions \ref{spin HS} and \ref{spin HS bis} separately.
\end{remark}

\subsubsection{The spin embedding $\varepsilon_\spin$ of the dual polar space $\mathrm{DQ}(2n,\KK)$}\label{sec spin DQ}

Let $\Delta = D_{n+1}(\KK)$, $V = V(2^n,\KK)$ and let $\varepsilon^+_\spin$ be the spin embedding of $\Delta_{n+1}$ in $\PG(V)$. As explained in Section \ref{3.2 Spin}, the embedding $\varepsilon^+_\spin$ induces the spin embedding $\varepsilon_\spin$ of $\mathrm{DQ}(2n,\KK)$ in $\PG(V)$. We recall that $|\varepsilon_\spin| = |\varepsilon^+_\spin|$.

\begin{prop}\label{spin DQ}
If either $\mathrm{char}(\KK) \neq 2$ or $\KK$ is infinite, then $\varepsilon_{\spin}$ is tightly $1$-opaque.
\end{prop}
\begin{proof}
The proof is similar to that of Proposition~\ref{spin HS}. We only give a sketch of it, leaving most of the details to the reader.

Put $\varepsilon := \varepsilon_\spin$ and $\hat{\varepsilon} := \varepsilon_\spin^+$, for short. Also, $\cS := \mathrm{Q}(2n,\KK)$ and $\cS_n := \mathrm{DQ}(2n,\KK)$. By the same arguments used to prove Corollary \ref{spin l1 cor} we obtain that $\varepsilon$ admits a tight degree of opacity, say $\chi$. Having proved this, we can exploit Lemma~\ref{cI} to prove that $\chi = 1$.

When $n = 2$ we have $\mathrm{DQ}(4,\KK) \cong\mathrm{W}(3,\KK)$ and $\varepsilon$ is isomorphic to the natural embedding of $\mathrm{W}(3,\KK)$ in $\PG(3,\KK)$, which is completely opaque. Clearly, $\chi = 1$ in this case.

Let $n > 2$ and let $\mathfrak{S}$ be the family of the $n$-grassmannians $R_n$, for $R$ a residue of $\cS$ of type $\{n-2, n-1, n\}$. Both conditions (S\ref{S1}) and (S\ref{S2}) of Section \ref{Sec 2.2} hold for this choice of $\mathfrak{S}$. In order to obtain the conclusion by Lemma \ref{cI} we only must prove that the embedding $\varepsilon_S$ induced by $\varepsilon$ on a member $S$ of $\mathfrak{S}$ is tightly $1$-opaque.

The hypotheses of Lemma \ref{l1} hold for the embedding $\varepsilon_S$. Hence $\varepsilon_S$ admits a tight degree of opacity, say $\chi_S$. It remains to prove that $\chi_S = 1$. We shall discuss this issue in some detail.

Note firstly that $S$ is convex and $\mathrm{diam}(S) = 3$ (see Section~\ref{Prel}). Let $A$ and $B$ be points of $S$ at distance $2$ and let $[A,B]$ be their convex hull. Then $[A,B]\cong \mathrm{W}(3,\KK)$ and the embedding $\varepsilon_{A,B}$ induced by $\varepsilon$ on $[A,B]$ is isomorphic to the natural embedding of $W(3,\KK)$ in $\PG(3,\KK)$. Hence $\varepsilon_{A,B}$ is completely opaque. Consequently, $\langle\varepsilon_S(A),\varepsilon_S(B)\rangle \subseteq |\varepsilon_S|$. Therefore $\chi_S \geq 1$.

In order to prove that $\chi_S = 1$ it only remains to prove that if $A$ and $B$ are points of $S$ at distance $3$ then $\langle\varepsilon_S(A),\varepsilon_S(B)\rangle\not\subseteq |\varepsilon_S|$. Suppose the contrary: $\langle\varepsilon_S(A),\varepsilon_S(B)\rangle\subseteq |\varepsilon_S|$. Then $\langle\hat{\varepsilon}(A),\hat{\varepsilon}(B)\rangle\subseteq|\hat{\varepsilon}|$, since $\hat{\varepsilon}$ induces $\varepsilon_S$ on $S$. However, $\hat{\varepsilon}$ is transparent (Proposition~\ref{spin HS}). Therefore $A$ and $B$ are collinear in $\Delta_{n+1}$. Let $L$ be the line of $\Delta_{n+1}$ through $A$ and $B$ and let $H$ be the hyperplane of $\PG(2n+1,\KK)$ which we have used to construct the polar space $\cS$ (see Subsection~\ref{DPP1}). The line $L$ corresponds to a singular subspace $X$ of $\Delta_1$ of rank $n-1$ and the points $A$ and $B$ are the intersections $A = H\cap A'$ and $B = H\cap B'$ of $H$ with uniquely determined singular subspaces $A'$ and $B'$ of $\Delta_1$ of  rank $n+1$ in the family corresponding to the type $n+1$ of $\Delta$. Clearly, $\rank(X\cap H) = n-2$ and $X$ contains the singular subspace $Y$ of $\cS$ of rank $n-3$ corresponding to the subspace $S$. Therefore $Z := X\cap H$ is a singular subspace of $\cS$ of rank $n-2$ and $Y \subset Z \subset A, B$. In other words, $A$ and $B$ are points of the subspace $S' := \sh_n(Z)$ of $\cS_n$. The latter has diameter $2$ and is contained in $S$.  Hence $A$ and $B$ have distance $2$ in $S$ too. This contradicts our choice of $A$ and $B$. Therefore $\langle\varepsilon_S(A),\varepsilon_S(B)\rangle\not\subseteq |\varepsilon_S|$, as claimed.
\end{proof}

\begin{prop}\label{spin DQ bis}
Let $\KK = \mathrm{GF}(q)$, with $q$ a power of $2$. Then $\varepsilon_\spin$ is tightly $1$-opaque.
\end{prop}
\begin{proof}
  Put $\varepsilon := \varepsilon_\spin$ and $\cS := \mathrm{Q}(2n,q)$. If $P, P'$ are points of $\cS_n$ at distance $2$ then their convex hull $[P,P']$ is isomorphic to $\mathrm{W}(3,q)$ and $\varepsilon$ induces on it its natural embedding in $\PG(3,q)$. Hence $\langle\varepsilon(P),\varepsilon(P')\rangle\subseteq|\varepsilon|$. It remains to prove that if $d(P,P') > 2$ then $\langle\varepsilon(P),\varepsilon(P')\rangle\not\subseteq|\varepsilon|$. This conclusion can be obtained by an inductive argument, as in the proof
  of Proposition~\ref{spin HS bis}. Leaving the details for the reader, we only recall the main steps of this argument. Arguing by contradiction as in the proof of Proposition~\ref{spin HS bis}, we are reduced to the case where $\langle\varepsilon(P),\varepsilon(P')\rangle\subseteq|\varepsilon|$ for any two points $P$ and $P'$ at maximal distance. The set $\mathrm{Far}(P)$ of the points of $\cS_n$ at maximal distance from $P$ admits a partition $\pi$ corresponding to the family of lines $\langle\varepsilon(P),\varepsilon(X)\rangle$ for $X\in \mathrm{Far}(P)$. The stabilizer $G_P$ of $P$ in the derived group $G$ of $\Aut(\cS)$ permutes the classes of $\pi$. Moreover, the unipotent radical $U$ of $G_P$ acts regularly on $\mathrm{Far}(P)$ and it is abelian, of order $|U| = q^{(n+1)n/2}$. As in the proof of Proposition~\ref{spin HS bis}, we can see that the class-wise stabilizer $U_0$ of $\pi$ in $U$ has order $q$ and acts regularly on each of the classes of $\pi$. Obviously, $U_0 \unlhd G_P$. At this stage, a contradiction can be obtained. This final step is worthy of a more careful discussion.

The group $U$ is isomorphic to a non-split extension $\mathrm{AM}_n(q)\cdot V(n,q)$ of the additive group $\mathrm{AM}_n(q)$ of alternating matrices of order $n$ over $\GF(q)$ by the additive group of $V(n,q)$, which can be described as follows. The elements of  $\mathrm{AM}_n(q)\cdot V(n,q)$ are the pairs $(A, v)$ with $A\in \mathrm{AM}_n(q)$ and $v\in V(n,q)$ and the multiplication is defined by the following rule: $(A,v)\cdot(B,w) = (A+B+vw^T+wv^T, v+w)$. If $C$ is a Levi complement of $U$ in $G_P$, then $C\cong \GL(n,q)$ acts as follows on $U$: a matrix $M\in C$ maps $(A,v)$ onto $(MAM^T, Mv)$. It is clear from this description that $U$ admits no subgroup $U_0$ of order $q$ normalized by $C$. We have reached a final contradiction.
 \end{proof}
 Propositions~\ref{spin DQ} and \ref{spin DQ bis} yield claim \ref{main2 pt2}
 of Theorem \ref{main thm 2}.

\begin{remark}\label{spin DQ rem}
The proof of Proposition~\ref{spin DQ bis} does not work for $q$ odd. Indeed when $\mathrm{char}(\KK) \neq 2$ the unipotent radical $U$ of $G_P$ is non-abelian. So, there is no way to prove that the class-wise stabilizer $U_0$ of the partition $\pi$ is non-trivial.

On the other hand, most of the proof of Proposition~\ref{spin DQ bis} also works for $\KK$ an infinite field of characteristic $2$. Most likely, the final part of that proof can be modified in such a way as to work in this case too, but we are not completely sure about this.
\end{remark}

\subsubsection{The spin embedding $\varepsilon^-_\spin$ of the dual polar space $\mathrm{DQ}^-(2n+1,\KK_\sigma)$}

Let $\Delta := D_{n+1}(\KK)$ and assume that $\KK$ admits an involutory non-trivial automorphism $\sigma$. Let $\KK_\sigma$ be the subfield of $\KK$ fixed by $\sigma$ and let $\delta\in \Aut(\Delta)$ be defined as in Subsection~\ref{DPP2}. As seen in that subsection, the polar space $\cS := \mathrm{Q}^-(2n+1,\KK_\sigma)$ is formed by the singular subspaces of $\Delta_1$ fixed by $\delta$. Accordingly, $\Aut(\cS) =  C_{\Aut(\Delta)_{1,n,n+1}}(\delta)  \cong  C_{\Aut(\Delta)_1}(\delta)/\langle\delta\rangle$. (Note that $C_{\Aut(\Delta)_1}(\delta) = C_{\Aut(\Delta)_{1,n,n+1}}(\delta)\times\langle\delta\rangle$.) The dual polar space $\cS_n$ is embedded in $\Delta_{n+1}$ in the most natural way: the points of $\cS_n$ bijectively correspond to the $\{n,n+1\}$-flags of $\Delta$ stabilized by $\delta$ and every such flag is mapped onto its $(n+1)$-element. Denoted by $\eta$ this embedding of $\cS_n$ in $\Delta_{n+1}$, the spin embedding $\varepsilon^-_\spin$ of $\cS_n$ is defined as the composition $\varepsilon^+_\spin\cdot\eta$ of $\eta$ with the spin embedding $\varepsilon^+_\spin$ of $\Delta_{n+1}$ (see Section~\ref{3.2 Spin}).

\begin{prop}\label{spin elliptic}
The spin embedding $\varepsilon^-_\spin$ is transparent.
\end{prop}
\begin{proof}
We know that $\varepsilon^+_\spin$ is transparent (Propositions~\ref{spin HS} and \ref{spin HS bis}). So, in order to prove that $\varepsilon^-_\spin$ is transparent we only must prove that $\eta(\cS_n)$ is a full subgeometry of $\Delta_{n+1}$, namely if a line $\ell'$ of $\Delta_{n+1}$ is fully contained in the $\eta$-image of the point-set of $\cS_n$, then $\ell' = \eta(\ell)$ for a line $\ell$ of $\cS_n$. Explicitly, we must prove the following:

\begin{itemize}
\item[$(*)$] If $A$ is an $(n-1)$-element of $\Delta$ such that $X$ is incident with $X^\delta$ in $\Delta$ for every $X\in \sh_{n+1}(A)$, then $A = A^\delta$.
\end{itemize}
As $C_{\Aut(\Delta)_1}(\delta)$ acts transitively on the set of $(n+1)$-elements $X$ of $\Delta$ such that $X^\delta$ is incident with $X$, we can assume  without loss of generality that $A$ has been chosen so that $\sh_{n+1}(A)$ contains an element $W:=\langle u_1, u_3,\ldots,u_{2n-1},u_{2n+1}\rangle$ such  that $W^{\delta}=\langle u_1, u_3,\ldots,u_{2n-1},u_{2n+2}\rangle$, where $U = (u_1, u_2,... , u_{2n+1}, u_{2n+2})$ is the basis of $V = V(2n+2,\KK)$ chosen in Subsection \ref{DPP2}.  With this set-up, $A = \langle a_1,\ldots,a_{n-1}\rangle\subset\langle u_1, u_3 \ldots,u_{2n-1},u_{2n+1}\rangle =W$.

Put $W_0:=W\cap W^{\delta}=\langle u_1, u_3,\ldots,u_{2n-1}\rangle$. Two cases must be considered. \\

\noindent
\emph{Case 1.} $A\subseteq W_0$. Since the stabilizer of $W_0$ in $\Aut(\cS)$ acts as $\PgL(n,\KK_\sigma)$ on $\mathrm{Res}_\cS(W_0)\cong \PG(n-1,\KK_\sigma)$, we can assume, up to a suitable change of basis for $A$, that there are $\lambda_i\in\KK$ such that $a_i=u_{2i-1}+\lambda_i u_{2n-1}$ for all $i=1,\ldots,n-1$. Claim $(*)$ now amounts to the following: if $\dim (X\cap X^{\delta})=n$ for any $(n+1)$-element $X$ of $\Delta$ with $A\subseteq X$, then all $\lambda_i$ are in $\KK_\sigma$.

When $\lambda_i=0$ for all $i=1,\ldots,n-1$ there is nothing to prove, as $a_i^{\delta}=u_{2i-1}^{\delta}=u_{2i-1}=a_i$. Suppose that there is $i$ such that $\lambda_i\neq 0$. Let $\lambda_1\neq 0$, to fix ideas. Define
\[w ~ :=~(x_1, x_2,\dots, x_{2n+1}, x_{2n+2}), \qquad w' ~:=~(x'_1, x'_2,\dots, x'_{2n+1},x'_{2n+2}).\]
Write $X=\langle a_1,\ldots,a_{n-1},w,w'\rangle$. By construction, we can assume that
\begin{equation}\label{31}
x_{2i-1} ~ = ~ x_{2i-1}' ~ =~0,\qquad \text{for } i=1,2,\ldots,n-1.
\end{equation}
The space $X$ is totally singular for the form $q(x_1, x_2,\ldots , x_{2n+1},x_{2n+2})$ associated with $\Delta_1$ (see Subsection~\ref{Cn-classical-field},  \eqref{form O+}) if and only if
\begin{equation}\label{32}
  \begin{cases}
x_{2i} =  -\lambda_ix_{2n},\qquad \text{for}~ i = 1, 2,\dots, n-1;\\
x'_{2i} =  -\lambda_ix'_{2n}, \qquad \text{for}~ i = 1, 2,\dots, n-1;\\
x_{2n-1}x_{2n} + x_{2n+1}x_{2n+2}  =  0;\\
x'_{2n-1}x'_{2n} + x'_{2n+1}x'_{2n+2} =  0;\\
x_{2n-1}x'_{2n} + x_{2n}x'_{2n-1} + x_{2n+1}x'_{2n+2} + x_{2n+2}x'_{2n+1}  = 0.
\end{cases}\end{equation}
Observe that if $w$ and $w'$ satisfy \eqref{31} and \eqref{32} and,
in addition, the set $\{a_1,\ldots,a_{n-1},w,w'\}$ is linearly independent,
then the matrix $M$ with rows given by
\[M ~= ~ [a_1,\dots, a_{n-1}, a_1^\delta,\dots, a_{n-1}^\delta, w, w', w^\delta, w'^\delta],\]
has rank $n+2$ (here each entry in the list is a distinct row vector containing
the components of the corresponding vector; so $M$ has $2n+2$ rows).
The condition $\lambda_1\in\KK_\sigma$ is equivalent to
$a_1=a_1^{\delta}$. So we just need to show that $M_0=[a_1,\dots, a_{n-1}, a_1^\delta,\dots, a_{n-1}^\delta]$ has rank $n-1$.

Since $\lambda_1\neq 0$, the $(3\times 3)$ minor of
$M$ encompassing rows $2n-1,2n,2n+1$, that is $[w,w',w^{\delta}]$, and columns $2,2n+1,2n+2$ is non-singular and
$x_{2n-1}=x_{2n-1}'=0$. So, the minor of $M$ determined by its first $n-1$ rows as well as rows $2n-1,2n,2n+1$ has rank $n+2$.
In particular, the remaining $n-1$ rows of $M$ must be linear combinations of these. Thus,
$a_1^{\delta}\in\langle a_1,\ldots,a_{n-1},w,w',w^{\delta}\rangle$.
This is possible only if $a_1^{\delta}=a_1$, that is $\lambda_1\in\KK_\sigma$. A similar argument for the remaining
values of $i$ completes the proof of $(*)$ in this case. \\

\noindent
\emph{Case 2.} $A\not\subseteq W_0$. We show that in this case
  there is an $X$ containing $A$ which is not incident with $X^{\delta}$.
  First, we can
  assume
  $a_1,\ldots,a_{n-2}\in U_0$ and write $a_{n-1}=b+u_{2n+1}$ for
  some $b\in W_0$.
  We can also assume
\[\begin{array}{ccl}
a_i & = &  u_{2i-1} + \lambda_i u_{2n-3} + \mu_i u_{2n-1}  ~~~ \mbox{for} ~ i = 1, 2,\dots, n-2,\\
a_{n-1} &  = & \lambda_{n-1}u_{2n-3} + \mu_{n-1}u_{2n-1} + u_{2n+1},
\end{array}\]
for suitable scalars $\lambda_1,\ldots, \lambda_{n-1}, \mu_1,\ldots, \mu_{n-1}$ in $\KK$. So,
\[\begin{array}{ccl}
a_i^\delta & = &  u_{2i-1} + \lambda^\sigma_i u_{2n-3} + \mu_i^\sigma u_{2n-1} ~~ \text{for} ~ i = 1, 2,\dots, n-2;\\
a_{n-1}^\delta &  = & \lambda_{n-1}^\sigma u_{2n-3} + \mu_{n-1}^\sigma u_{2n-1} + u_{2n+2}.
\end{array}\]
If we take two vectors $w$ and $w'$ with $x_{2i-1} = x'_{2i-1}$ for $i = 1, 2,\ldots, n-1$ as in Case 1, the conditions for $X=\langle a_1,\ldots,a_{n-1},w,w'\rangle$
to be totally singular for the form associated with $\Delta_1$ become
\begin{equation}\label{42} \begin{cases}
x_{2i} =  -\lambda_ix_{2n-2} - \mu_ix_{2n}, \qquad \text{for}~ i = 1, 2,\dots, n-2;\\
x'_{2i} =  -\lambda_ix'_{2n-2} - \mu_ix'_{2n}, \qquad \text{for}~ i = 1, 2,\dots, n-2;\\
x_{2n+2} = -\lambda_{n-1}x_{2n-2} - \mu_{n-1}x_{2n};  \\
x'_{2n+2} = -\lambda_{n-1}x'_{2n-2} - \mu_{n-1}x'_{2n};  \\
x_{2n-3}x_{2n-2} + x_{2n-1}x_{2n}  =  0;\\
x'_{2n-3}x'_{2n-2} + x'_{2n-1}x'_{2n}  =  0;\\
x_{2n-3}x'_{2n-2}+ x_{2n-2}x'_{2n-3} + x_{2n-1}x'_{2n} + x_{2n}x'_{2n-1}  = 0.
\end{cases}\end{equation}
We now need to show that there are $w$, $w'$ such that \eqref{42} is
fulfilled, the vectors $\{a_1,\ldots,a_{n-1},w,w'\}$ are linearly independent and the matrix
\[M ~= ~ [a_1,\dots, a_{n-1}, a_1^\delta,\dots, a_{n-1}^\delta, w, w', w^\delta, w'^\delta]\]
has rank larger than $n+2$. To this aim, observe that the minor of $M$
given by the first $n-1$ rows together with row $2n-2$ as well as columns
$1,3,5,\ldots,2n-5,2n+1,2n+2$ has full rank $n$. So it is enough to choose
$w$ and $w'$ such that the minor of $M$ comprising the last $4$ rows and
columns $2n-3,2n-1,2n-2$ and $2n$ has rank $3$. This is always possible. Claim $(*)$ is proved in this case too.
\end{proof}
Proposition \ref{spin elliptic} corresponds to claim~\ref{main2 pt4} of Theorem~\ref{main thm 2}.

\section{More on $\cS_k$ with $\cS = \mathrm{W}(2n-1,\KK)$}
\label{Sec6}

Symplectic $k$-grassmannians with $k < n$ are not considered in Corollary~\ref{co} because their Pl\"{u}cker embeddings are not transparent. However, the following holds.

\begin{theorem}
\label{qt}
Let $\cS := \mathrm{W}(2n-1,\KK)$ and, for $1 < k < n$, let $\varepsilon := \varepsilon_k^\cS$ be the Pl\"{u}cker embedding of the $k$-grassmannian $\cS_k$ of $\cS$. Then $\aut(|\varepsilon|) = \Aut(\varepsilon)$.
\end{theorem}
\begin{proof}
  Let $V = V(2n,\KK)$, $\cG = \PG(V)$, $\cG_k$ the $k$-grassmannian of $\cG$ and $\hat{\varepsilon} := \varepsilon_k^\cG$ the Pl\"{u}cker embedding of $\cG_k$, as in
  Section~\ref{Sec4p}. Also, let $f$ be the alternating form of $V$ associated to $\cS$. Denoted by $\perp$ the orthogonality relation with respect to $f$, a subspace $X$ of $\cG$ is totally isotropic for $f$ (namely it belongs to $\cS$) if and only if $X\subseteq X^\perp$.

It is well known that the Grassmann variety $|\hat{\varepsilon}|$ contains two families of maximal projective subspaces, say $\widehat{\cM}^+$ and $\widehat{\cM}^-$, defined as follows; see Blok and Cooperstein \cite{PGPG}.
In order to keep this work self-contained, we recall their characterization
here. Given two subspaces $Z_1$ and $Z_2$ of $\cG$ of rank $k-1$ and $k+1$ respectively, define
  \[\begin{array}{rcl}
\widehat{M}^+(Z_1) & := & \{ \hat{\varepsilon}(X) ~\colon ~ \mbox{$X$ subspace of $\cG$ of rank $k$}, ~ X \supset Z_1 \},\\
\widehat{M}^-(Z_2) & := & \{ \hat{\varepsilon}(X) ~\colon ~ \mbox{$X$ subspace of $\cG$ of rank $k$}, ~ X \subset Z_2 \}.
\end{array} \]
Then $\widehat{\cM}^+:=\{\widehat{M}^+(Z_1)\colon \rank(Z_1)=k-1 \}$ and $\widehat{\cM}^-:=\{\widehat{M}^-(Z_2)\colon \rank(Z_2)=k+1 \}$.

It is well known that $|\varepsilon|$ spans a subspace $\Sigma := \langle|\varepsilon|\rangle$ of $\PG(\bigwedge^kV)$ of codimension ${{2n}\choose {k-2}}$. Moreover $|\varepsilon|=|\hat{\varepsilon}|\cap\Sigma$ (see e.g. \cite[Section 4]{P16}). Hence every maximal projective subspace of $|\varepsilon|$ is the intersection of $\Sigma$ with a suitable maximal projective subspace of $|\hat{\varepsilon}|$.

Explicitly, consider $M^+(Z_1) := \widehat{M}^+(Z_1)\cap\Sigma$. If $Z_1\not\subseteq Z_1^\perp$ then $M^+(Z_1) = \emptyset$. Otherwise, the points of $M^+(Z_1)$ are the subspaces of $\cG$ of rank $k$ which contain $Z_1$ and are totally isotropic for $f$. These subspaces bijectively correspond to the points of the polar spaces associated with the alternating form induced by $f$ on $Z_1^\perp/Z_1$. The points of the latter polar space are just the points of $\PG(Z_1^\perp/Z_1)$ and $\dim(\PG(Z_1^\perp/Z_1)) = \dim(Z_1^\perp)-\dim(Z_1)-1 = (2n-k)-(k-2)-1 = 2n-2k+1$. Therefore $\dim(M^+(Z_1)) = 2n-2k+1$.

Consider now $M^-(Z_2) := \widehat{M}^-(Z_2)\cap\Sigma$. If $Z_2\subset Z_2^\perp$ then $\widehat{M}^-(Z_2)\subseteq\Sigma$. In this case $M^-(Z_2) = \widehat{M}^-(Z_2)$ and $\dim(M^-(Z_2)) = k$.

On the other hand, suppose $Z_2^\perp\cap Z_2 \subset Z_2$. The space $Z_2^\perp\cap Z_2$ has even codimension in $Z_2$, say $\mathrm{cod}_{Z_2}(Z_2^\perp\cap Z_2) = 2r$ for a positive integer $r$. If $r = 1$ then $M^-(Z_2)$ is a line of $\Sigma$, but not the $\varepsilon$-image of a line of $\cS_k$. In this case $M^-(Z_2) \subseteq M^+(Z_1)$ with $Z_1 = Z_2^\perp\cap Z_2$. The space $M^+(Z_1)$ has dimension $2n-2k+1 > 1$ (as $k < n$ by assumption). Hence $M^-(Z_2)$ is not maximal among the projective subspaces of $|\varepsilon|$. Finally, let $r > 1$. The maximal subspaces of $Z_2$ totally isotropic for $f$ have rank $\rank(Z_2)-r = k+1-r < k$. In this case $M^-(Z_2) = \emptyset$.

To sum up, $|\varepsilon|$ admits only the following two families of maximal projective subspaces:
\[\begin{array}{rcl}
{\cM}^+ & := & \{M^+(Z_1) = \widehat{M}^+(Z_1)\cap\Sigma~\colon~ \rank(Z_1)=k-1, ~Z_1\subseteq Z_1^\perp\},\\
{\cM}^- & := & \{M^-(Z_ 2) = \widehat{M}^-(Z_2) \subseteq \Sigma ~\colon ~\rank(Z_2)=k+1, ~ Z_2\subseteq Z_2^\perp\}.
\end{array} \]
The members of $\cM^+$ and $\cM^-$ have projective dimensions equal to $2n-2k+1$ and $k$ respectively. Moreover, if $\ell = \sh_k(\{Z_1, Z_2\})$ is a line of $\cS_k$,  for a $\{k-1,k+1\}$-flag $\{Z_1, Z_2\}$ of $\cS$, then $\varepsilon(\ell) = M^+(Z_1)\cap M^-(Z_2)$. Conversely, if $M^+(Z_1)\in \cM^+$ and $M^-(Z_2)\in \cM^-$, then $M^+(Z_1)\cap M^-(Z_2)$ is either empty or a line of $\Sigma$ contained in $|\varepsilon|$. In the latter case, $M^+(Z_1)\cap M^-(Z_2) = \varepsilon(\ell)$ where $\ell = \sh_k(\{Z_1, Z_2\})$ (a line of $\cS_k$).

So, if some property exists which enables us to distinguish $\cM^+$ from $\cM^-$ in the family $\cM := \cM^+\cup\cM^-$ of maximal projective subspaces of $|\varepsilon|$, then we can recover the $\varepsilon$-images of the lines of $\cS_k$ as the nonempty intersections of members of $\cM^+$ with members of $\cM^-$. Having done this, the equality $\aut(|\varepsilon|) = \Aut(\varepsilon)$ follows.

We shall now explain how to recognize the subfamilies $\cM^+$ and $\cM^-$ in $\cM$. As remarked above, the members of $\cM^+$ and $\cM^-$ have projective dimensions equal to $2n-2k+1$ and $k$ respectively. If $2n-2k+1 \neq k$, then dimensions are sufficient to distinguish the members of $\cM^+$ from those of $\cM^-$. In this case we are done.

Suppose that $2n-2k+1 = k$. Two members of the same family $\cM^+$ or $\cM^-$ always have at most one point in common. So, we can define a connected bipartite graph $\mathfrak{M}$ with $\cM$ as the set of vertices and adjacency defined as follows: two subspaces $S, S'\in \cM$ are adjacent in
$\mathfrak{M}$ precisely when $S\cap S'$ is a projective line. Thus, $\cM^+$ and $\cM^-$ are the two classes of the bipartition of $\mathfrak{M}$ (uniquely determined because $\mathfrak{M}$ is connected). It remains to understand which is which of these two classes. In view of this, note that every line of a subspace $M^-(Z_2)\in \cM^-$ is contained in a member of $\cM^+$. On the other hand, if $M^+(Z_1) \in \cM^+$ then the lines of $M^+(Z_1)$ that also belong to members of $\cM^-$ are the $\varepsilon$-images of the lines $\sh_k(\{Z_1,Z_2\})$ of $\cS_k$, for some singular subspace $Z_2$ of $\cS$ of rank $k+1$ containing $Z_1$. These are precisely the projective lines of $Z_1^\perp/Z_1$ that are totally isotropic with respect to the form $f_{Z_1}$ induced by $f$ on $Z_1^\perp/Z_1$. However, the full set of lines of $M^+(Z_1)$ is the set of all projective lines of $Z_1^\perp/Z_1$. Not all of them are totally isotropic for $f_{Z_1}$. So, not all lines of a $M^+(Z_1)$ belong to members of $\cM^-$. This is enough to distinguish $\cM^+$ from $\cM^-$.
\end{proof}



\end{document}